\documentclass[12pt]{amsart}
\usepackage[english]{babel}
\usepackage{amsmath}
\usepackage{amsthm}
\usepackage{anysize}
\usepackage{mathtools}
\usepackage{mathrsfs}
\usepackage{amssymb}
\usepackage{xcolor}
\usepackage{xfrac}
\usepackage{esint}
\usepackage{graphicx}
\usepackage{float}
\usepackage{array}
\usepackage{enumitem}
\usepackage{tikz-cd}
\usepackage{centernot}
\usepackage{stmaryrd}
\usepackage{comment}
\excludecomment{confidential}
\usepackage{graphicx,color}
\usepackage{tikz,pgfplots,pgf}
\usepackage[font=small]{caption}

\newtheorem{theorem}{Theorem}[section]
\newtheorem{lemma}{Lemma}[section]

\newtheorem{corollary}{Corollary}[section]
\theoremstyle{definition}

\theoremstyle{definition}

\newtheorem{definition}{Definition}[section]

\usepackage[english]{babel}
\usepackage{amsmath}
\usepackage{amsthm}
\usepackage{mathtools}
\usepackage{mathrsfs}
\usepackage{amssymb}
\usepackage{xcolor}
\usepackage{xfrac}
\usepackage{esint}
\usepackage{graphicx}
\usepackage{float}
\usepackage{array}
\usepackage{enumitem}
\usepackage[new]{old-arrows}
\usepackage[all,cmtip]{xy}
\usepackage{tikz-cd}
\usepackage{centernot}
\usepackage{stmaryrd}
\usepackage{comment}
\usepackage{bbm}

\newcommand{\mc}{\mathcal}
\newcommand{\mf}{\mathfrak}

\newcommand{\R}{\mathbb{R}}
\newcommand{\N}{\mathbb{N}}
\newcommand{\C}{\mathbb{C}}

\renewcommand{\l}{\lambda}

\DeclareMathOperator{\colim}{colim}
\DeclareMathOperator{\spann}{span}

\numberwithin{equation}{section}

\newcommand{\normord}[1]{:\mathrel{#1}:}

\makeatletter
\@namedef{subjclassname@2020}{\textup{2020} Mathematics Subject Classification}
\makeatother

\begin{document}

\title{On the $L^{p}$-spaces of projective limits of probability measures}

\author{Juan Carlos Sampedro} \thanks{The author has been supported by the Research Grant PID2021--123343NB--I00 of the Spanish Ministry of Science, Technology and Universities}
\address{Institute of Interdisciplinary Mathematics (IMI), Madrid, Spain.
	Universidad Polit\'ecnica de Madrid, Departamento de Matem\'atica Aplicada a la Ingenier\'ia Industrial, Ronda de Valencia 3, 28012 Madrid, Spain}
\email{juancarlos.sampedro@upm.es}


\begin{abstract}
	The present article describes the precise structure of the $L^{p}$-spaces of projective limit measures by introducing a category theoretical perspective. This analysis is applied to measures on vector spaces and in particular to Gaussian measures on nuclear topological vector spaces. A simple application to constructive Quantum Field Theory (QFT) is given through the Osterwalder--Schrader axioms.
\end{abstract}

\maketitle
\tableofcontents

\section{Introduction}\label{section-1}

The notion of projective limit of probability spaces is fundamental in various fields of mathematics. For instance, in probability theory, the existence of stochastic processes is usually equivalent to the existence of such a limit through the celebrated Kolmogorov extension theorem \cite{K}. Moreover, the construction of various types of  measures on infinite dimensional vector spaces used in constructive Quantum Field Theory (QFT) is done through this concept, see for instance Yamasaki \cite[Ch. 3]{Y}, Glimm--Jaffe \cite[Ch. 6]{GJ} and Simon \cite{Si}. The existence of projective limits of probability spaces has been extensively studied since the foundational works of Kolmogorov \cite{K} and Bochner \cite{B} in the fifties. The greatest efforts in previous works have been devoted to establishing necessary conditions on the projective system for the existence of the projective limit measure (see, e.g., Choksi \cite{Ch}, Metivier \cite{Me}, Bourbaki \cite{Bou}, Mallory and Sion \cite{Ma}, Rao \cite{RA}, Frolík \cite{F}, Rao and Sazonov \cite{RS} and Pintér \cite{P}, among others). In this article we will change the perspective slightly. Assuming the existence of the projective limit measure, we describe the $L^{p}$-spaces of such measure by introducing a category theoretical perspective. The analogue analysis for the inductive limit measure was partially established by Macheras \cite{M}. However the projective case is more interesting from the point of view of applications where projective limits are more present that inductive ones. 
\par This paper is organized as follows. Section \ref{S2} introduces the basic notions of categorical limits that will make the article self-contained for non-experts in category theory. Section \ref{Se3} presents the main results of the article. Among the main findings, we prove that the $L^{p}$-spaces of the projective limit $(X_{\infty},\mu_{\infty})$ of the projective system $(X_{i},\mu_{i})_{i\in I}$, where $I$ is a given directed set, can be described as an explicit quotient of a subspace of the product space $\prod_{i\in I}L^{p}(X_{i},\mu_{i})$. This result implies that the integration of functions on $(X_{\infty},\mu_{\infty})$ can be reduced to limits of the type
 \begin{equation*}
	\int_{X_{\infty}}f \, \mathrm{d}\mu_{\infty}=\lim_{i\in I}\int_{X_{i}}f_{i}\, \mathrm{d}\mu_{i},
\end{equation*}
where the limit is understood in the sense of net convergence. The author is referred to Section \ref{Se3} for exact definitions of the notations and concepts involved in these results. Section \ref{Se4} applies the abstract techniques of Section \ref{Se3} to study the measures on vector spaces. This measures are commonly defined through a projective limit measure and consequently the analysis of this article is naturally applied to this case. Section \ref{Se4} is organized into two subsections. Subsection \ref{S4.1} focusses on measures on topological vector spaces. In Subsection \ref{S4.2} we obtain analogous results for Gaussian measures on nuclear spaces. This kind of measures are central in constructive (Euclidean) QFT through the program started by E. Nelson \cite{Ne}, J. Glimm and A. Jaffe \cite{GJ0,GJ} and F. Guerra, L. Rosen and B. Simon \cite{GRS} in the seventies through the Osterwalder--Schrader axioms. As a by-product, in Section \ref{S4.3}, using the techniques developed along Section \ref{Se4}, we give an application to the computation of Schwinger (or correlation) functions through the limit of finite dimensional integrals. In the final Section \ref{Se6}, we give an application of the theory developed along Section \ref{Se3} to obtain certain unitary representations. We set out to answer the following question. Given a family of unitary representations on $L^{2}(X_{i},\mu_{i})$, can we construct an unitary representation on $L^{2}(X_{\infty},\mu_{\infty})$?
\par As a result of the research carried out for this paper, we have obtained a general result to describe the colimit of codiagrams in the category of Banach spaces $\mf{Ban}$. For the convenience of the reader interested in category theory, we include this result in the Appendix \ref{Se5}.

\section{Basic notions regarding categorical limits}\label{S2}

In this section we proceed to describe briefly the basic notions of projective limits (limits) and inductive limits (colimits) in category theoretical language. Recall that a \textit{directed set} is a nonempty set $I$ together with a binary relation $\leq$ which is both  reflexive and transitive (that is, a preorder) and for any two elements $i, j \in I$ there is a $k\in I$ such that $i \leq k$ and $j \leq k$. Any directed set may be regarded as a category. The elements of $I$ are the objects of the category and given $i,j\in I$, there exists a unique morphism $i\to j$ if and only if $i\leq j$.
Let $\mathfrak{C}$ be a fixed category and $I$ a directed set. A \textit{diagram} of shape $I$ is a contravariant functor $\mf{F}:I\to \mathfrak{C}$. It can be represented as $(X_{i},\varphi_{ij})$ for a family of objects $\{X_{i}:i\in I\}$ of $\mf{C}$ indexed by $I$, and for each $i\leq j$ a morphism $\varphi_{ij}:X_{j}\to X_{i}$ such that $\varphi_{ii}$ is the identity on $X_{i}$ and $\varphi_{ik}=\varphi_{ij}\circ \varphi_{jk}$ for each $i\leq j\leq k$. In general, we can consider diagrams as functors $\mf{F}:\mathfrak{J}\to\mathfrak{C}$ indexed by a general category $\mathfrak{J}$. However, diagrams indexed by directed sets are enough for our purposes.
A \textit{cone} of $\mf{F}$ is a pair $(X,\phi_{i})$ where $X$ is an object of $\mf{C}$ and $\phi_{i}:X\to X_{i}$ is a morphism of $\mf{C}$ such that $\phi_{i}=\varphi_{ij}\circ \phi_{j}$ for every $i\leq j$. Let us denote by $\mf{K}(\mf{F})$ the category of cones of $\mf{F}$, that is, the category whose objects are cones $(X,\phi_{i})$ and given two cones $(X,\phi_{i})$, $(Y,\psi_{i})$, the morphisms $(X,\phi_{i})\to (Y,\psi_{i})$ are just morphisms $\varphi: X\to Y$ of $\mf{C}$ such that the diagram of Figure \ref{F0} commutes.

\begin{figure}[h!]
	\[
	\xymatrix@C+1em@R+1em{ 
		& X \ar^{\varphi}[d]  \ar@/_1em/_{\phi_j}[ddl] \ar@/^1em/^{\phi_i}[ddr] & \\
		& Y \ar^{\psi_i}[dr] \ar_{\psi_j}[dl] & \\
		X_j \ar^{\varphi_{ij}}[rr]  & & X_i
	}
	\]
	\caption{Diagram I}
	\label{F0}
\end{figure}
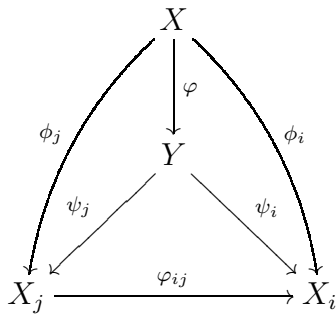

A \textit{limit} of the diagram $(X,\varphi_{ij})$ is defined to be a cone $(Y,\psi_{i})$ characterizing by the following universal property: For every other cone $(X,\phi_{i})$, there exists a unique morphism $\varphi:X\to Y$ making the diagram of Figure \ref{F0} commutative. 
The limit of a given diagram can also be characterized as the terminal object in the category of cones of $\mf{F}$, $\mf{K}(\mf{F})$. The limit of a given diagram does not necessarily exist, however if it does exist, it is unique up to a unique isomorphism in the category of cones. Any given isomorphism class of the limit is called a \textit{realization}. The limit of a given diagram $\mf{F}:I\to \mf{C}$, $\mf{F}\equiv(X_{i},\varphi_{ij})$, is commonly denoted by $\lim \mf{F}= (\lim X_{i},\psi_{i})$.

Dually, a \textit{codiagram} is a covariant functor $\mf{F}:I\to \mathfrak{C}$. It can be represented as $(X_{i},\varphi_{ij})$ for a family of objects $\{X_{i}:i\in I\}$ of $\mf{C}$ indexed by $I$, and for each $i\leq j$ a morphism $\varphi_{ij}:X_{i}\to X_{j}$ such that $\varphi_{ii}$ is the identity on $X_{i}$ and $\varphi_{ik}=\varphi_{jk}\circ \varphi_{ij}$ for each $i\leq j\leq k$. A \textit{cocone} is a pair $(X,\phi_{i})$ where $X$ is an object of $\mathfrak{C}$ and $\phi_{i}:X_{i}\to X$ is a morphism of $\mathfrak{C}$ such that $\phi_{i}=\phi_{j}\circ \varphi_{ij}$ for every $i\leq j$. 
Let us denote by $\mf{K}^{\ast}(\mf{F})$ the category of cocones of $\mf{F}$, that is, the category whose objects are cocones $(X,\phi_{i})$ and given two cocones $(X,\phi_{i})$, $(Y,\psi_{i})$, the morphisms $(Y,\psi_{i})\to(X,\phi_{i})$ are just the morphisms $\varphi: Y\to X$ of $\mf{C}$ such that the diagram of Figure \ref{F20} commutes.

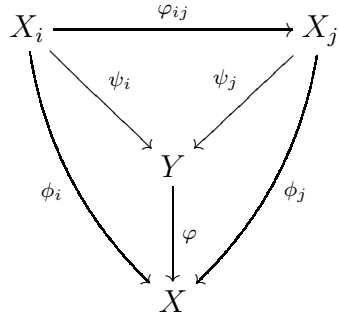
\begin{figure}[h!]
	\[
	\xymatrix@C+1em@R+1em{ 
		X_i \ar^{\varphi_{ij}}[rr] \ar^{\psi_i}[dr] \ar@/_1em/_{\phi_i}[ddr] & & X_j \ar_{\psi_j}[dl] \ar@/^1em/^{\phi_j}[ddl] \\
		&Y \ar^{\varphi}[d] & \\
		& X &
	}
	\]
	\caption{Diagram II}
	\label{F20}
\end{figure}

\noindent A \textit{colimit} of the codiagram $(X_{i},\varphi_{ij})$ is defined to be a cocone $(Y, \psi_{i})$ characterized by the following universal property: For any cocone $(X,\phi_{i})$, there exists a unique morphism $\varphi:Y\to X$ making the diagram of Figure \ref{F20} commutative.

Analogously, the colimit of a given codiagram can also be characterized as the initial object in the category of cocones $\mf{K}^{\ast}(\mf{F})$. The colimit of a given codiagram does not necessarily exist, however if it does exist, it is also unique up to a unique isomorphism in the category of cocones. The colimit of a given codiagram $\mf{F}:I\to \mf{C}$, $\mf{F}\equiv(X_{i},\varphi_{ij})$, is commonly denoted by $\colim \mf{F}= (\colim X_{i},\psi_{i})$. 
\par We refer the reader to Mc Lane \cite[Ch. III]{ML} and Riehl \cite[Ch. 3]{ER} for further details about categorical limits.
\par In this article, we will consider the category $\mathfrak{Ban}$ whose objects are Banach spaces and whose morphisms are linear isometries. For a deep study of this category and its colimits, we refer the reader to Wegge-Olsen \cite[App. L]{WO} and Castillo \cite{JC} and references therein.
\par Throughout this article, we will make use of the following simple lemma.
\begin{lemma}
	\label{L2.1}
	Let $(X_{i},\varphi_{ij})$ be a codiagram on $\mathfrak{Ban}$. Then a cocone $(Y,\psi_{i})$ is a realization of the colimit of $(X_{i},\varphi_{ij})$ if and only if $\bigcup_{i\in I}\psi_{i}(X_{i})$ is dense in $Y$.
\end{lemma}

\begin{proof}
	Suppose that $(Y,\psi_{i})$ defines a realization of the colimit of $(X_{i},\varphi_{ij})$. Consider the cocone $(X,\rho_{i})$ where
	\begin{equation*}
	X:=\overline{\bigcup_{i\in I}\psi_{i}(X_{i})}\subset Y,
	\end{equation*}
	and where $\rho_{i}:X_{i}\to X$ are defined by $\rho_{i}(x_{i})=\psi_{i}(x_{i})$ for each $x_{i}\in X_{i}$. Note that $X$ is a Banach space since $\bigcup_{i\in I}\psi_{i}(X_{i})$ is a vector subspace of $Y$. Indeed, let $y_{i}=\psi_{i}(x_{i})$ and $y_{j}=\psi_{j}(x_{j})$ with $i\leq j$. Then, a simple computation shows 
	\begin{align*}
		k_{1}y_{i}+k_{2}y_{j}=\psi_{i}(k_{1} x_{i})+\psi_{j}(k_{2} x_{j}) &= (\psi_{j}\circ \varphi_{ij})(k_{1}x_{i})+\psi_{j}(k_{2} x_{j}) \\
		&=\psi_{j}(\varphi_{ij}(k_{1}x_{i})+k_{2}x_{j})\in \bigcup_{i\in I}\psi_{i}(X_{i}),
	\end{align*}
	for each $k_{1}, k_{2}\in \R$. By the universal property, there exists a unique isometry $\varphi:Y\to X$ making the diagram of Figure \ref{F20} commutative. It is straightforward to prove that the morphism $\varphi$ is in fact an isomorphism. Hence $\bigcup_{i\in I}\psi_{i}(X_{i})$ is dense in $Y$. On the other hand, take a cocone $(Y,\psi_{i})$ such that $\bigcup_{i\in I}\psi_{i}(X_{i})$ dense in $Y$. For any other cocone $(X,\rho_{i})$, define the morphism $\varphi:Y\to X$ by 
	\begin{equation*}
	\varphi\circ\psi_{i}:=\rho_{i}, \quad i\in I.
	\end{equation*}
	It is easy to see that the morphism $\varphi$ is the unique morphism making the diagram of Figure \ref{F20} commutative. Hence $(Y,\psi_{i})$ is a colimit.
\end{proof}

\section{Abstract results}\label{Se3}

In this section we present the main findings of the paper. Let us start by fixing some notation. As usual, we denote a measurable space by a pair $(X,\mc{B})$ where $X$ is the main set and $\mc{B}$ is a $\sigma$-algebra of subsets of $X$.  Let $\mathfrak{M}_{0}$ be the category whose objects are measurable spaces $(X,\mathcal{B})$ and whose morphisms are measurable maps $f:(X,\mathcal{B})\to (Y,\mathcal{S})$. Given a directed set $I$, a diagram $\mf{F}:I\to\mathfrak{M}_{0}$ can be described by a family $(X_{i},\mathcal{B}_{i},\pi_{ij})$ where $(X_{i},\mathcal{B}_{i})$ are measurable spaces and 
$$\pi_{ij}:(X_{j},\mathcal{B}_{j})\longrightarrow (X_{i},\mathcal{B}_{i}), \quad i\leq j,$$ 
are morphisms in $\mathfrak{M}_{0}$ such that $\pi_{ii}$ is the identity and $\pi_{ik}=\pi_{ij}\circ\pi_{jk}$, for each $i\leq j\leq k$. The limit of any diagram $(X_{i},\mathcal{B}_{i},\pi_{ij})$ exists and can be described by the realization $(X_{\infty},\mathcal{B}_{\infty},\pi_{i})$ where 
\begin{equation*}
X_{\infty}:=\Big\{(x_{i})_{i\in I}\in \prod_{i\in I}X_{i} \, : \, \pi_{ij}(x_{j})=x_{i} \text{ for each } i\leq j \Big\},
\end{equation*}
the morphisms $\pi_{j}:X_{\infty}\to X_{j}$ are the projection maps given by $\pi_{j}(x_{i})_{i\in I}:=x_{j}$ and
\begin{equation*}
\mathcal{B}_{\infty}:=\sigma(\{\pi_{i}^{-1}(B) \, : \, B\in\mathcal{B}_{i}, \ i\in I\}).
\end{equation*}
Given a family $\mc{A}$ of subsets of a given set $X$, the notation $\sigma(\mc{A})$ stands for the $\sigma$-algebra of subsets of $X$ generated by $\mc{A}$.
Let us denote by $\mathfrak{M}$ the category whose objects are probability spaces $(X,\mathcal{B},\mu)$ and whose morphisms are measure preserving maps $f:(X,\mathcal{B},\mu)\to (Y,\mathcal{S},\nu)$, i.e., measurable maps that satisfy $f_{\ast}\mu=\nu$, where $f_{\ast}\mu:\mc{S}\to[0,+\infty]$ is the push-forward measure defined by  
$$f_{\ast}\mu(B):=\mu(f^{-1}(B)), \quad B\in \mc{S}.$$
Given a directed set $I$, a diagram $\mf{F}:I\to\mathfrak{M}$ can be described by a family $(X_{i},\mathcal{B}_{i},\mu_{i},\pi_{ij})$ where $(X_{i},\mathcal{B}_{i},\mu_{i})$ are probability spaces and 
$$\pi_{ij}:(X_{j},\mathcal{B}_{j},\mu_{j}) \longrightarrow (X_{i},\mathcal{B}_{i},\mu_{i}), \quad i\leq j,$$
are morphisms in $\mathfrak{M}$ such that $\pi_{ii}$ is the identity and $\pi_{ik}=\pi_{ij}\circ\pi_{jk}$, for each $i\leq j\leq k$. 
\begin{definition}
	A diagram $(X_{i},\mathcal{B}_{i},\mu_{i},\pi_{ij})$ is said to be \textit{convergent} if there exists a measure $\mu_{\infty}$ on the limit $(X_{\infty},\mc{B}_{\infty},\pi_{i})$ of $(X_{i},\mc{B}_{i},\pi_{ij})$ in $\mf{M}_{0}$ such that 
$$(\pi_{i})_{\ast}\mu_{\infty}=\mu_{i} \ \ \text{for each } i\in I.$$ 
\end{definition}
Every convergent diagram $(X_{i},\mathcal{B}_{i},\mu_{i},\pi_{ij})$ has a limit in $\mf{M}$ and it is described by the realization $(X_{\infty},\mathcal{B}_{\infty},\mu_{\infty},\pi_{i})$. The characterization of convergent diagrams on $\mf{M}$ has taken a great deal of work, we refer the reader to the classical result in Yamasaki \cite[Th. 7.2, Part A]{Y} for some sufficient conditions for the convergence. A sharper result has been obtained recently by Pintér \cite{P}. 
\par Subsequently, when we cite a probability space we will omit to write the associated $\sigma$-algebra if this is clear from the context.
For each $1\leq p<\infty$, there is a contravariant functor
\begin{equation*}
L^{p}:\mathfrak{M}\longrightarrow \mathfrak{Ban},
\end{equation*}
that assign to every object $(X,\mu)$ in $\mathfrak{M}$, the object $L^{p}(X,\mu)$ in $\mathfrak{Ban}$ and to every morphism $f:(X,\mu)\to (Y,\nu)$ in $\mathfrak{M}$ the morphism $f^{\ast}:L^{p}(Y,\nu)\to L^{p}(X,\mu)$ in $\mf{Ban}$ defined by $f^{\ast}(u):=u\circ f$ for each $u\in L^{p}(Y,\nu)$. Indeed, the morphism $f^{\ast}:L^{p}(Y,\nu)\to L^{p}(X,\mu)$ is an isometry since the change of variable formula for measures (see for instance Cohn \cite[Pr. 2.6.5]{C}) and the identity $f_{\ast}\mu=\nu$ yield
$$\|f^{\ast}(u)\|_{L^{p}(\mu)}^{p}=\int_{X}|u\circ f|^{p}  \, \mathrm{d}\mu=\int_{Y}|u|^{p}  \, \mathrm{d}f_{\ast}\mu=\int_{Y}|u|^{p} \, \mathrm{d}\nu=\|u\|_{L^{p}(\nu)}^{p}.$$
Our first result states that the $L^{p}$-functor preserves limits of convergent diagrams.

\begin{theorem}
	\label{T3.1}
	Let $(X_{i},\mu_{i},\pi_{ij})$ be a convergent diagram in $\mathfrak{M}$ with limit $(X_{\infty},\mu_{\infty},\pi_{i})$. Then for each $1\leq p<\infty$, the cocone $(L^{p}(X_{\infty},\mu_{\infty}),\pi^{\ast}_{i})$ is the colimit of the codiagram $(L^{p}(X_{i},\mu_{i}),\pi_{ij}^{\ast})$ in $\mf{Ban}$. Symbolically,
	$$\colim \, (L^{p}(X_{i},\mu_{i}),\pi_{ij}^{\ast}) = (L^{p}(X_{\infty},\mu_{\infty}),\pi^{\ast}_{i}).$$
	In other words, the $L^{p}$ functor preserves convergent limits with shape a directed set $I$, i.e., for any convergent diagram $\mf{F}:I\to \mathfrak{M}$, 
	\begin{equation*}
	L^{p}(\lim \mf{F})=\colim L^{p}\mf{F}.
	\end{equation*}
\end{theorem}

\begin{proof}
	Let us consider the family of subsets of $X_{\infty}$ given by
	\begin{equation*}
	\mathcal{R}:=\{\pi_{i}^{-1}(B) \, : \, B\in\mathcal{B}_{i}, \ i\in I\}.
	\end{equation*}
	A routine calculation shows that $\mathcal{R}$ is an algebra of subsets. Moreover, by definition, $\mathcal{B}_{\infty}=\sigma(\mathcal{R})$. By Cohn \cite[Le. 3.4.6]{C} it follows that
	\begin{equation*}
	\text{span}\{\mathbbm{1}_{R}:R\in\mathcal{R}\}
	\end{equation*}
	is dense in $L^{p}(X_{\infty},\mu_{\infty})$ where $\mathbbm{1}_{R}:X_{\infty}\to\R$ stands for the characteristic function of the subset $R\in\mc{R}$. Given $R=\pi^{-1}_{i}(B)$ for some $i\in I$ and $B\in\mathcal{B}_{i}$, we can write
	\begin{equation*}
	\mathbbm{1}_{R}=\mathbbm{1}_{B}\circ\pi_{i},
	\end{equation*}
	and from this we obtain the set  inclusion
	\begin{equation}
		\label{dens}
	\text{span}\{\mathbbm{1}_{R}:R\in\mathcal{R}\}\subset\bigcup_{i\in I}\pi^{\ast}_{i}(L^{p}(X_{i},\mu_{i})).
	\end{equation}
	Consequently, the space $\bigcup_{i\in I}\pi^{\ast}_{i}(L^{p}(X_{i},\mu_{i}))$ is dense in $L^{p}(X_{\infty},\mu_{\infty})$ and by Lemma \ref{L2.1} the cocone $(L^{p}(X_{\infty},\mu_{\infty}),\pi^{\ast}_{i})$ is the colimit of $(L^{p}(X_{i},\mu_{i}),\pi_{ij}^{\ast})$.
\end{proof}

We proceed to identify the space $L^{p}(X_{\infty},\mu_{\infty})$ intrinsically in terms of the spaces $L^{p}(X_{i},\mu_{i})$. Let $(X_{i},\mu_{i},\pi_{ij})$ be a convergent diagram in $\mathfrak{M}$ with limit $(X_{\infty},\mu_{\infty},\pi_{i})$. A sequence $(f_{i})_{i\in I}\in \prod_{i\in I}L^{p}(X_{i},\mu_{i})$ is said to be \textit{co-Cauchy} if for every $\varepsilon>0$, there exists $\ell\in I$ such that 
\begin{equation}
\label{coC}
\|\pi_{ij}^{\ast}(f_{i})-f_{j}\|_{L^{p}(\mu_{j})}<\varepsilon, \quad \forall (i,j)\in I\times I,  \quad \ell \leq i\leq j.
\end{equation}
This notion was first introduced in Sampedro \cite{S} for the special case where $\mu_{\infty}$ is the Wiener measure.  Let $\mf{F}:I\to\mf{M}$ be the functor representing the diagram $(X_{i},\mu_{i},\pi_{ij})$. We consider the product space
\begin{equation*}
\mathscr{L}^{p}(\mf{F}):=\Big\{(f_{i})_{i\in I}\in \prod_{i\in I}L^{p}(X_{i},\mu_{i}) \, : \, (f_{i})_{i\in I} \text{ is co-Cauchy}\Big\}\Big\slash \sim,
\end{equation*}
where two sequences $(f_{i})_{i\in I}, (g_{i})_{i\in I}\in \prod_{i\in I}L^{p}(X_{i},\mu_{i})$ are related, $(f_{i})_{i\in I}\sim(g_{i})_{i\in I}$, if by definition,
\begin{equation*}
\lim_{i\in I} \, \|f_{i}-g_{i}\|_{L^{p}(\mu_{i})}=0,
\end{equation*}
where the limit is understood in the sense of net convergence. 
\par Recall that a \textit{net} in $\R$, $\mathfrak{x}=(x_{i})_{i\in I}$, is a function $\mf{x}: I \to \R$, whose domain $I$ is some directed set. We usually write $\mf{x}(i)=x_{i}$ for all $i\in I$. A net $(x_{i})_{i\in I}$ in $\mathbb{R}$ \textit{converges} to $x\in \R$ if for all $\varepsilon>0$ there exists $\ell\in I$ so that whenever $\ell \leq i$ we have $|x_{i}-x|<\varepsilon$. In this case, we write $\lim_{i\in I} x_{i}=x$. If it exists, the limit is unique. The notion of Cauchy net is defined analogously and a net in $\R$ converges if and only if it is Cauchy. 
\par By the elementary properties of the limit, it becomes apparent that $\sim$ is an equivalence relation. We define a norm on $\mathscr{L}^{p}(\mf{F})$ by
\begin{equation*}
\|(f_{i})_{i\in I}\|_{\mathscr{L}}:=\lim_{i\in I} \, \|f_{i}\|_{L^{p}(\mu_{i})}, \quad (f_{i})_{i\in I}\in \mathscr{L}^{p}(\mf{F}).
\end{equation*}
The convergence of the net $(\|f_{i}\|_{L^{p}(\mu_{i})})_{i\in I}$ follows from the co-Cauchy property. Indeed, given $\varepsilon>0$, choose $\ell\in I$ such that \eqref{coC} holds for $\varepsilon/2$ and take $i,r\geq \ell$. Since $I$ is directed, there exists $j\in I$ such that $i,r\leq j$ and consequently 
$$\ell\leq i\leq j, \quad \ell\leq r\leq j.$$
Hence, by \eqref{coC} we infer that
\begin{equation}
		\label{Eqp}
\begin{split}
\left|\|f_{i}\|_{L^{p}(\mu_{i})}-\|f_{r}\|_{L^{p}(\mu_{r})}\right|&=\left|\|\pi^{\ast}_{ij}(f_{i})\|_{L^{p}(\mu_{j})}-\|\pi^{\ast}_{rj}(f_{r})\|_{L^{p}(\mu_{j})}\right|\\
&\leq \| \pi^{\ast}_{ij}(f_{i})-\pi^{\ast}_{rj}(f_{r})\|_{L^{p}(\mu_{j})}\\
&\leq \|\pi_{ij}^{\ast}(f_{i})-f_{j}\|_{L^{p}(\mu_{j})}+\|\pi^{\ast}_{rj}(f_{r})-f_{j}\|_{L^{p}(\mu_{j})}<\varepsilon.
\end{split}
\end{equation}
This shows that the net $(\|f_{i}\|_{L^{p}(\mu_{i})})_{i\in I}$ is Cauchy and hence convergent. The next result identifies the space $\mathscr{L}^{p}(\mf{F})$ with $	L^{p}(X_{\infty},\mu_{\infty})$.
\begin{theorem}
	\label{T3.2}
	Let $(X_{i},\mu_{i},\pi_{ij})$ be a convergent diagram in $\mathfrak{M}$ with limit $(X_{\infty},\mu_{\infty},\pi_{i})$. Then for each $1\leq p<\infty$, the $L^{p}$-space $L^{p}(X_{\infty},\mu_{\infty})$ is isometrically isomorphic to $\mathscr{L}^{p}(\mf{F})$. Symbolically, 
	\begin{equation*}
	L^{p}(X_{\infty},\mu_{\infty})\simeq\mathscr{L}^{p}(\mf{F}).
	\end{equation*}
\end{theorem}

\begin{proof}
	For $1\leq p<\infty$, let us consider the operator
	\begin{equation*}
	\mf{I}_{p}:\mathscr{L}^{p}(\mf{F})\longrightarrow L^{p}(X_{\infty},\mu_{\infty}), \quad (f_{i})_{i\in I}\mapsto \lim_{i\in I} \pi_{i}^{\ast}(f_{i}).
	\end{equation*}
	We start by showing that the operator $\mf{I}_{p}$ is well defined. Let $(f_{i})_{i\in I}\in\mathscr{L}^{p}(\mf{F})$. For any given $\varepsilon>0$, taking $\ell\in I$ satisfying \eqref{coC}, we obtain
	\begin{align*}
	\|\pi^{\ast}_{i}(f_{i})-\pi_{j}^{\ast}(f_{j})\|_{L^{p}(\mu_{\infty})}&=\|(\pi^{\ast}_{j}\circ\pi^{\ast}_{ij})(f_{i})-\pi_{j}^{\ast}(f_{j})\|_{L^{p}(\mu_{\infty})}\\
	&=\|\pi_{ij}^{\ast}(f_{i})-f_{j}\|_{L^{p}(\mu_{j})}<\varepsilon, \quad \forall (i,j)\in I\times I,  \quad \ell \leq i\leq j.
	\end{align*}
	Hence $(\pi_{i}^{\ast}(f_{i}))_{i\in I}$ is a Cauchy net in $L^{p}(X_{\infty},\mu_{\infty})$, from which follows its convergence. On the other hand, given two related elements $(f_{i})_{i\in I}\sim (g_{i})_{i\in I}$, we have $\mf{I}_{p}(f_{i})_{i\in I}=\mf{I}_{p}(g_{i})_{i\in I}$ since
	\begin{align*}
	\|\mf{I}_{p}(f_{i})_{i\in I}-\mf{I}_{p}(g_{i})_{i\in I}\|_{L^{p}(\mu_{\infty})}&=\lim_{i\in I} \, \|\pi_{i}^{\ast}(f_{i})-\pi^{\ast}_{i}(g_{i})\|_{L^{p}(\mu_{\infty})}\\
	&=\lim_{i\in I} \, \|f_{i}-g_{i}\|_{L^{p}(\mu_{i})}=0.
	\end{align*}
	This proves that $\mf{I}_{p}$ is well defined. By an analogous computation we obtain that $\mf{I}_{p}$ is an isometry:
	\begin{equation*}
	\|\mf{I}_{p}(f_{i})_{i\in I}\|_{L^{p}(\mu_{\infty})}=\lim_{i\in I} \, \|\pi_{i}^{\ast}(f_{i})\|_{L^{p}(\mu_{\infty})}=\lim_{i\in I} \, \|f_{i}\|_{L^{p}(\mu_{i})}=\|(f_{i})_{i\in I}\|_{\mathscr{L}}.
	\end{equation*}
	Finally, we prove that the inverse of $\mf{I}_{p}$ is given by
	\begin{equation*}
	\mf{R}_{p}: L^{p}(X_{\infty},\mu_{\infty})\longrightarrow \mathscr{L}^{p}(\mf{F}), \quad f\mapsto ([\pi_{i}^{\ast}]^{-1}(\mathbb{E}_{i}(f)))_{i\in I},
    \end{equation*}
    where $\mathbb{E}_{i}(f):=\mathbb{E}[f \, | \, \mc{B}^{i}]$ is the conditional expectation of $f\in L^{p}(X_{\infty},\mu_{\infty})$ conditioned by the $\sigma$-algebra 
    $$\mc{B}^{i}:=\{\pi_{i}^{-1}(B):B\in\mc{B}_{i}\}, \quad i\in I.$$
	Let us show that $\mf{R}_{p}$ is well defined. Let $\mc{B}_{\R}$ be the Borel $\sigma$-algebra of $\R$ with the usual topology. Since the maps $\pi_{i}:(X_{\infty},\mc{B}^{i})\to (X_{i},\mc{B}_{i})$, $\mathbb{E}_{i}(f): (X_{\infty},\mc{B}^{i})\to (\R,\mc{B}_{\R})$ are measurable, by the Doob–Dynkin lemma (see for instance Kallenberg \cite[Le. 1.14]{OlKa}), there exists a measurable function $g_{i}:(X_{i},\mc{B}_{i})\to (\R,\mc{B}_{\R})$ such that $\mathbb{E}_{i}(f)=g_{i}\circ \pi_{i}$. In other words, there exists a measurable function $g_{i}:(X_{i},\mc{B}_{i})\to (\R,\mc{B}_{\R})$ making the diagram of Figure \ref{F990} commutative.

		\begin{figure}[h!]
			\begin{tikzcd}[row sep=large, column sep=large]
				(X_{\infty},\mc{B}^{i}) \arrow{r}{\mathbb{E}_{i}(f)} \arrow{d}{\pi_{i}} & \R \\
			(X_{i},\mc{B}_{i}) \arrow{ur}{g_{i}} & 
			\end{tikzcd}
			\caption{Diagram III}
			\label{F990}
		\end{figure}

	\noindent By the change of variable formula for measures, we obtain
	$$\int_{X_{\infty}}|\mathbb{E}_{i}(f)|^{p}\, \mathrm{d}\mu=\int_{X_{\infty}}|g_{i}\circ \pi_{i}|^{p} \, \mathrm{d}\mu=\int_{X_{i}}|g_{i}|^{p} \, \mathrm{d}(\pi_{i})_{\ast}\mu_{\infty}=\int_{X_{i}}|g_{i}|^{p} \, \mathrm{d}\mu_{i}.$$
	Consequently $g_{i}\in L^{p}(X_{i},\mu_{i})$ and $\mathbb{E}_{i}(f)=\pi_{i}^{\ast}(g_{i})$. This implies that $$[\pi_{i}^{\ast}]^{-1}(\mathbb{E}_{i}(f))=g_{i}\in L^{p}(X_{i},\mu_{i}), \quad i\in I.$$ 
	Let $\varepsilon>0$. Since by \eqref{dens} the subspace $\bigcup_{i\in I}\pi^{\ast}_{i}(L^{p}(X_{i},\mu_{i}))$ is dense in $L^{p}(X_{\infty},\mu_{\infty})$, there exits $\ell\in I$ and $h_{\ell}\in \pi^{\ast}_{\ell}(L^{p}(X_{\ell},\mu_{\ell}))$, such that
	$$\|h_{\ell}-f\|_{L^{p}(\mu_{\infty})}<\frac{\varepsilon}{2}.$$
	Since $\mc{B}^{\ell}\subset \mc{B}^{i}$ for each $\ell\leq i$, if $h_{\ell}$ is $\mc{B}^{\ell}$-measurable, then $h_{\ell}$ is also $\mc{B}^{i}$-measurable.
	By the properties of the conditional expectation, this implies that $\mathbb{E}_{i}(h_{\ell})=h_{\ell}$ for each $\ell\leq i$.
	Consequently, we infer that
	\begin{align*}
		\|\mathbb{E}_{i}(f)-f\|_{L^{p}(\mu_{\infty})}&\leq \|\mathbb{E}_{i}(f)-\mathbb{E}_{i}(h_{\ell})\|_{L^{p}(\mu_{\infty})}+\|h_{\ell}-f\|_{L^{p}(\mu_{\infty})} \\
		&\leq 2 \, \|h_{\ell}-f\|_{L^{p}(\mu_{\infty})}<\varepsilon.
	\end{align*}
    The last inequality follows from the fact that the conditional expectation is a norm one projection. This proves that 
    \begin{equation}
    	\label{Eq.3}
    \lim_{i\in I} \ \mathbb{E}_{i}(f)=f \quad \text{ in } L^{p}(X_{\infty},\mu_{\infty}),
    \end{equation}
    and therefore $([\pi_{i}^{\ast}]^{-1}(\mathbb{E}_{i}(f)))_{i\in I}$ is co-Cauchy. Hence $\mf{R}_{p}$ is well-defined. Finally by \eqref{Eq.3} we deduce on the one hand,
    \begin{align*}
     	(\mf{I}_{p}\circ\mf{R}_{p})(f)=\lim_{i\in I} \, \pi_{i}^{\ast}[\pi_{i}^{\ast}]^{-1}(\mathbb{E}_{i}(f))=\lim_{i\in I}  \, \mathbb{E}_{i}(f)=f,
    \end{align*}
and on the other hand,
    \begin{align*}
    	(\mf{R}_{p}\circ \mf{I}_{p})(f_{i})_{i\in I}&=\left([\pi_{j}^{\ast}]^{-1}\left[\mathbb{E}_{j} \, (\lim_{i\in I} \, \pi_{i}^{\ast}(f_{i}))\right]\right)_{j\in I} =\left([\pi_{j}^{\ast}]^{-1}\left[\lim_{i\in I}\, \mathbb{E}_{j}( \pi_{i}^{\ast}(f_{i}))\right]\right)_{j\in I} \\
    	&=\left([\pi_{j}^{\ast}]^{-1}\pi_{j}^{\ast}(f_{j})\right)_{j\in I} = (f_{j})_{j\in I}.
    \end{align*}
The equality in the second line holds in the quotient. Let $F=\lim_{i\in I}\pi^{\ast}_{i}(f_{i})$. Then, we have
\begin{align*}
\|[\pi^{\ast}_{j}]^{-1}\mathbb{E}_{j}(F)-f_{j}\|_{L^{p}(\mu_{j})}
&=\|\mathbb{E}_{j}(F-\pi^{\ast}_{j}(f_{j}))\|_{L^{p}(\mu_\infty)}\\
&\leq\|F-\pi^{\ast}_{j}(f_{j})\|_{L^{p}(\mu_\infty)}\longrightarrow0.
\end{align*}

	This proves that $\mf{I}_{p}^{-1}=\mf{R}_{p}$ and concludes the proof.
\end{proof}
To conclude this story, it would be convenient to give another realization of the colimit of the codiagram $(L^{p}(X_{i},\mu_{i}),\pi_{ij}^{\ast})$ in terms of the space $\mathscr{L}^{p}(\mf{F})$. In Theorem \ref{T3.2} we proved that $L^{p}(X_{\infty},\mu_{\infty})\simeq\mathscr{L}^{p}(\mf{F})$. By Theorem \ref{T3.1}, as $(L^{p}(X_{\infty},\mu_{\infty}),\pi^{\ast}_{i})$ is the colimit of the codiagram $(L^{p}(X_{i},\mu_{i}),\pi_{ij}^{\ast})$, we proceed to define the precise morphisms $(\psi_{i})_{i\in I}$, $\psi_{i}:L^{p}(X_{i},\mu_{i})\to \mathscr{L}^{p}(\mf{F})$, in such a way that $(\mathscr{L}^{p}(\mf{F}),\psi_{i})$ is isomorphic to $(L^{p}(X_{\infty},\mu_{\infty}),\pi^{\ast}_{i})$ in the category of cocones $\mf{K}^{\ast}(L^{p}\mf{F})$. This is enough to prove that $(\mathscr{L}^{p}(\mf{F}),\psi_{i})$ defines another realization of the colimit of $(L^{p}(X_{i},\mu_{i}),\pi_{ij}^{\ast})$. Symbolically,
$$\colim \, (L^{p}(X_{i},\mu_{i}),\pi_{ij}^{\ast}) =  (\mathscr{L}^{p}(\mf{F}),\psi_{i}).$$
The isometries $(\psi_{i})_{i\in I}$, $\psi_{i}:L^{p}(X_{i},\mu_{i})\to \mathscr{L}^{p}(\mf{F})$ are precisely given by
\begin{equation*}
\psi_{i}(f_{i})=(\pi^{\ast}_{ij}(f_{i}))_{j\geq i}, \quad (\pi^{\ast}_{ij}(f_{i}))_{j\geq i}:=\left\{\begin{array}{ll}
\pi^{\ast}_{ij}(f_{i}) & \text{ if } i\leq j \\
0 & \text{ else.}
\end{array}\right.
\end{equation*}
It is easy to see that $(\mathscr{L}^{p}(\mf{F}),\psi_{i})$ defines a cocone in $\mathfrak{Ban}$.
\begin{theorem}
	\label{T3.3}
	Let $(X_{i},\mu_{i},\pi_{ij})$ be a convergent diagram in $\mathfrak{M}$ with limit $(X_{\infty},\mu_{\infty},\pi_{i})$. Then for each $1\leq p<\infty$, the cocone $(\mathscr{L}^{p}(\mf{F}),\psi_{i})$ defines a realization of the colimit of $(L^{p}(X_{i},\mu_{i}),\pi_{ij}^{\ast})$. 
\end{theorem}
\begin{proof}
	Let $L^{p}\mf{F}:I\to\mf{Ban}$ be the functor associated to the codiagram $(L^{p}(X_{i},\mu_{i}),\pi_{ij}^{\ast})$. Since the colimit is unique up to a isomorphism on the category of cocones $\mf{K}^{\ast}(L^{p}\mf{F})$, it is enough to prove that the cocone $(\mathscr{L}^{p}(\mf{F}),\psi_{i})$ is isomorphic to $(L^{p}(X_{\infty},\mu_{\infty}),\pi^{\ast}_{i})$ in $\mf{K}^{\ast}(L^{p}\mf{F})$. By the proof of Theorem \ref{T3.2}, the map 
	\begin{equation*}
	\mathfrak{I}_{p}:\mathscr{L}^{p}(\mf{F})\longrightarrow L^{p}(X_{\infty},\mu_{\infty}), \quad (f_{i})_{i\in I}\mapsto \lim_{i\in I} \, \pi^{\ast}_{i}(f_{i}),
	\end{equation*}
	is an isometric isomorphism with inverse
		\begin{equation*}
		\mf{I}^{-1}_{p}: L^{p}(X_{\infty},\mu_{\infty})\longrightarrow \mathscr{L}^{p}(\mf{F}), \quad f\mapsto ([\pi_{i}^{\ast}]^{-1}(\mathbb{E}_{i}(f)))_{i\in I}.
	\end{equation*}
	Therefore the morphism $\mf{I}_{p}$ defines an isomorphism in $\mf{Ban}$. To prove that it defines an isomorphism of cocones, we need to verify the commutativity of the diagram of Figure \ref{F99}, as well as the commutativity of the same diagram replacing the arrow of $\mf{I}_{p}$ with the opposite arrow $\mf{I}_{p}^{-1}$.

		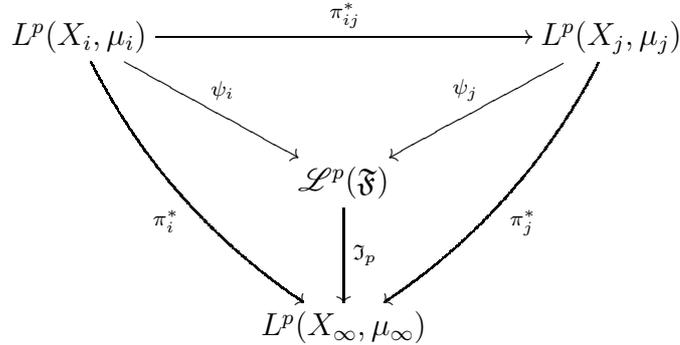
\begin{figure}[h!]
			\[
			\xymatrix@C+1em@R+1em{ 
				L^{p}(X_{i},\mu_{i}) \ar^{\pi^{\ast}_{ij}}[rr] \ar^{\psi_{i}}[dr] \ar@/_1em/_{\pi^{\ast}_{i}}[ddr] & & L^{p}(X_{j},\mu_{j}) \ar_{\psi_{j}}[dl] \ar@/^1em/^{\pi^{\ast}_{j}}[ddl] \\
				&\mathscr{L}^{p}(\mf{F}) \ar^{\mathfrak{I}_{p}}[d] & \\
				&  L^{p}(X_{\infty},\mu_{\infty}) &
			}
			\]
			\caption{Diagram IV}
			\label{F99}
		\end{figure}

	\noindent Take $f_{i}\in L^{p}(X_{i},\mu_{i})$, then by a simple computation we obtain
	\begin{equation*}
	(\mathfrak{I}_{p}\circ \psi_{i})(f_{i})=\lim_{r}h_{r}, \quad h_{r}:=\left\{
	\begin{array}{ll}
	\pi^{\ast}_{i}(f_{i}) & \text{ if } i\leq r \\
	0 & \text{ else, } 
	\end{array}
	\right.
	\end{equation*}
	and thus $(\mathfrak{I}_{p}\circ \psi_{i})(f_{i})=\pi^{\ast}_{i}(f_{i})$. Hence the diagram of Figure \ref{F99} commutes. Conversely, given $f_{i}\in L^{p}(X_{i},\mu_{i})$, on the tail $j\geq i$ we have
	\begin{equation*}
		(\mf{I}_{p}^{-1}\circ\pi_{i}^{\ast})(f_{i})=\left([\pi_{j}^{\ast}]^{-1}\mathbb{E}_{j}( \pi_{i}^{\ast}(f_{i}))\right)_{j\geq i}=\left([\pi_{j}^{\ast}]^{-1}[ \pi_{i}^{\ast}(f_{i})]\right)_{j\geq i}=(\pi_{ij}^{\ast}(f_{i}))_{j\geq i},
	\end{equation*}
    where the last equality follows from the identity $\pi_{j}^{\ast}\circ\pi_{ij}^{\ast}=\pi^{\ast}_{i}$ which is satisfied for each $i\leq j$.
	Thus $(\mf{I}_{p}^{-1}\circ\pi_{i}^{\ast})(f_{i})=\psi_{i}(f_{i})$ and the proof is completed.
\end{proof}
As a rather direct application of the isometric property of $\mathfrak{I}_{p}$, we obtain a integral limit approximation. Indeed, if $f\in L^{p}(X_{\infty},\mu_{\infty})$, there exists an element $(f_{i})_{i\in I}\in\mathscr{L}^{p}(\mf{F})$ such that $\|f\|_{L^{p}(\mu_{\infty})}=\|(f_{i})_{i}\|_{\mathscr{L}}$. Therefore, we have
\begin{equation*}
\int_{X_{\infty}}|f|^{p}\, \mathrm{d}\mu_{\infty}=\lim_{i\in I}\int_{X_{i}}|f_{i}|^{p}\, \mathrm{d}\mu_{i}.
\end{equation*}
\noindent Moreover, this element can be explicitly given by 
$$f_{i}:=[\pi_{i}^{\ast}]^{-1}(\mathbb{E}_{i}(f))\in L^{p}(X_{i},\mu_{i}), \quad i\in I.$$
Furthermore, if we take into account that given $f\in L^{1}(X_{\infty},\mu_{\infty})$, we can rewrite $f=f^{+}-f^{-}$ with $f^{+}, f^{-}$ positive and $f^{+},f^{-}\in L^{1}(X_{\infty},\mu_{\infty})$,  we infer the existence of $(f_{i})_{i\in I}\in\prod_{i\in I}L^{1}(X_{i},\mu_{i})$ such that 
	\begin{equation*}
	\int_{X_{\infty}}f \, \mathrm{d}\mu_{\infty}=\lim_{i\in I}\int_{X_{i}}f_{i}\, \mathrm{d}\mu_{i}.
	\end{equation*}
Moreover, the sequence $(f_{i})_{i\in I}$ can be chosen to be $f_{i}=[\pi_{i}^{\ast}]^{-1}(\mathbb{E}_{i}(f))$ for each $i\in I$.

\section{Measures on vector spaces}\label{Se4}

In this section we apply the abstract setting of Section \ref{Se3} to measures on vector spaces. Let $E$ be a real vector space. We denote the algebraic dual of $E$ by $E^{a}$, i.e., $E^{a}$ is the space consisted on linear functionals $\ell:E\to \R$. There is a natural embedding of $E$ into $[E^{a}]^{a}$ given by $\xi\mapsto \xi_{\ast}$, where $\xi_{\ast}$ is the linear functional on $E^{a}$ defined by $\ell\mapsto \ell(\xi)$ for each $\ell\in E^{a}$. We introduce the $\sigma$-algebra $\mc{B}_{E}$ on $E^{a}$ through
$$\mc{B}_{E}:=\sigma(\{\xi^{-1}_{\ast}(B) \, : \, \xi\in E, \ B\in\mc{B}_{\R}\}),
$$
where $\mc{B}_{\R}$ is the Borel $\sigma$-algebra of $\R$. We denote by $\mc{R}$ the class of finite dimensional subspaces of $E$. It forms a directed set under the inclusion. 
\begin{definition}[Yamasaki, \cite{Y}, Def. 16.1]
A map $\chi: E \to \C$ with $\chi(0)=1$ is called a \textit{characteristic function} if the following conditions are satisfied: 
\begin{enumerate}
\item \textit{Positively defined}: For each integer $n\in \N$, vectors $\xi_{1},\xi_{2},\cdots, \xi_{n}\in E$ and scalars $\alpha_{1},\alpha_{2},\cdots,\alpha_{n}\in\C$, the following positivity condition holds:
$$
\sum_{j,k=1}^{n}\alpha_{j}\bar{\alpha}_{k}\chi(\xi_{j}-\xi_{k})\geq 0.
$$
\item \textit{Continuity}: For each $R\in \mc{R}$, the restricted function $\chi|_{R}:R\to\C$ is continuous, where we are considering the usual Euclidean topology on $R$. 
\end{enumerate}
\end{definition}
Given a probability measure $\mu$ on $(E^{a},\mc{B}_{E})$, the map $\chi:E\to \C$ defined by
\begin{equation}
	\label{E22}
	\chi(\xi)=\int_{E^{a}}e^{i\ell(\xi)} \, \mathrm{d}\mu(\ell), \quad \xi\in E,
\end{equation}
defines a characteristic function on $E$. If $E$ is finite dimensional, Bochner theorem (see, e.g., Yamasaki \cite[Th. 16.1]{Y})  states the converse, that is, to every characteristic function $\chi$, there corresponds a unique Borel measure $\mu$ on $E^{a}$ satisfying equation \eqref{E22}.
This result can be generalized to infinite dimensions via the Bochner--Minlos theorem \cite[Th. 16.2]{Y}. Indeed, to every characteristic function $\chi$, there corresponds a unique measure $\mu$ on the measurable space $(E^{a},\mc{B}_{E})$ satisfying \eqref{E22}. Let us give an idea of the proof of this result as it will be useful for further considerations. For every finite dimensional subspace $R\in\mc{R}$ of  $E$, the restricted characteristic function $\chi|_{R}:R\to\C$ is a finite dimensional characteristic function. By Bochner theorem, there exists a Borel measure $\mu_{R}:\mc{B}_{R}\to[0,+\infty]$ on $R^{a}$ satisfying \eqref{E22}. Then, the triple $(R^{a},\mu_{R},\pi_{RF})$, where for each $R,F\in\mc{R}$ with $R\leq F$, $\pi_{RF}$ are the measurable functions defined by
$$
\pi_{RF}:(F^{a},\mu_{F})\longrightarrow (R^{a},\mu_{R}), \quad \ell_{F}\mapsto \ell_{F}|_{R}, \quad R\leq F,$$
is a diagram on $\mf{M}$ indexed by the directed set $\mc{R}$. We represent this diagram via the functor $\mf{R}:\mc{R}\to \mf{M}$. According to Yamasaki \cite[Th. 15.1]{Y}, the diagram $\mf{R}$ is convergent with limit $(\mc{R}_{\infty},\mc{B}_{\infty},\mu_{\infty},\pi_{R})$, where
$$
\mc{R}_{\infty}:=\Big\{(\ell_{R})_{R}\in \prod_{R\in\mc{R}} R^{a} \, : \, \pi_{RF}(\ell_{F})=\ell_{R}, \ R\leq F\Big\},
$$
and the maps $\pi_{F}:\mc{R}_{\infty}\to F^{a}$ are defined by $\pi_{F}(\ell_{R})_{R}:=\ell_{F}$ for each $F\in\mc{R}$.
This space is equipped with the $\sigma$-algebra $\mc{B}_{\infty}$ defined through
$$ \mc{B}_{\infty}:=\sigma(\{\pi_{R}^{-1}(B) \, : \, B\in \mc{B}_{R}, \ R\in \mc{R}\}),$$
and $\mu_{\infty}:\mc{B}_{\infty}\to [0,+\infty]$ is a measure satisfying $(\pi_{R})_{\ast}\mu_{\infty}=\mu_{R}$ for each $R\in\mc{R}$. On a further step, it is shown that the map $$I:(E^{a},\mc{B}_{E})\longrightarrow (\mc{R}_{\infty},\mc{B}_{\infty}), \quad \ell\mapsto (\ell|_{R})_{R\in\mc{R}},$$
is an isomorphism in $\mf{M}_{0}$ with inverse
$$I^{-1}:(\mc{R}_{\infty},\mc{B}_{\infty})\longrightarrow(E^{a},\mc{B}_{E}), \quad (\ell_{R})_{R\in\mc{R}}\mapsto \ell,$$
where $\ell\in E^{a}$ is defined by $\ell(\xi):=\ell_{R}(\xi)$ where $R:=\text{span}\{\xi\}$.
Hence, defining the measure 
$$\mu:\mc{B}_{E}\longrightarrow[0,+\infty], \quad \mu(B):=\mu_{\infty}(I(B)),$$ we obtain that $I:(E^{a},\mu)\to (\mc{R}_{\infty},\mu_{\infty})$ is an isomorphism in $\mf{M}$. It can be shown that this measure $\mu$ satisfies \eqref{E22}. We refer the reader to Yamasaki \cite[\S 15--16]{Y} for more details concerning the construction of the measure $\mu$.
\par We introduce the cone $(E^{a},\mu,\varphi_{R})$ of $\mf{R}$, where for each $R\in\mc{R}$, the morphism $\varphi_{R}:E^{a}\to R^{a}$ is given by 
\begin{equation}
	\label{Hom3}
\varphi_{R}:E^{a}\longrightarrow R^{a}, \quad \varphi_{R}(\ell):=\ell|_{R}.
\end{equation}
An easy computation shows that 
$$I:(E^{a},\mu,\varphi_{R})\longrightarrow (\mc{R}_{\infty},\mu_{\infty},\pi_{R}),$$
 is an isomorphism of cones in $\mf{K}(\mf{R})$. Moreover since functors preserve cone isomorphisms, we obtain that the cocones $(L^{p}(E^{a},\mu),\varphi^{\ast}_{R})$ and $(L^{p}(\mc{R}_{\infty},\mu_{\infty}),\pi^{\ast}_{R})$ are isomorphic in the category $\mf{K}^{\ast}(L^{p}\mf{R})$ via the application of the $L^{p}$ functor to $I$:
$$I^{\ast}:(L^{p}(\mc{R}_{\infty},\mu_{\infty}),\pi^{\ast}_{R})\longrightarrow (L^{p}(E^{a},\mu),\varphi^{\ast}_{R}), \quad f\mapsto f\circ I.$$ 
We start by applying the results of Section \ref{Se3} to this particular case.
\begin{theorem}
	Let $E$ be a real vector space, $\chi:E\to \C$ a characteristic function and $\mu$ be the corresponding measure on $(E^{a},\mc{B}_{E})$. Then the cone $(E^{a},\mu,\varphi_{R})$ is a realization of the limit of the convergent diagram $(R^{a},\mu_{R},\pi_{RF})$. Moreover for each $1\leq p<\infty$,
	$$\colim \ (L^{p}(R^{a},\mu_{R}),\pi^{\ast}_{RF})= (L^{p}(E^{a},\mu), \varphi^{\ast}_{R}).$$ 
\end{theorem}
\begin{proof}
	By the above construction, the cone $(\mc{R}_{\infty},\mu_{\infty},\pi_{R})$ is the limit of $(R^{a},\mu_{R},\pi_{RF})$. Since the morphism $I: (E^{a},\mu,\varphi_{R})\to (\mc{R}_{\infty},\mu_{\infty},\pi_{R})$ is a cone isomorphism and the limit is unique (up to isomorphisms) in the category of cones $\mf{K}(\mf{R})$, it follows that $(E^{a},\mu,\varphi_{R})$ defines a realization of the limit of $(R^{a},\mu_{R},\pi_{RF})$. On the other hand, Theorem \ref{T3.1} applied to the diagram $(R^{a},\mu_{R},\pi_{RF})$ assures that $(L^{p}(E^{a},\mu),\varphi^{\ast}_{R})$ is the colimit of the codiagram $(L^{p}(R^{a},\mu_{R}),\pi^{\ast}_{RF})$.
\end{proof}
Let us consider for each $1\leq p < \infty$, the normed space $\mathscr{L}^{p}(\mf{R})$. For the convenience of the reader, let us recall the definition of $\mathscr{L}^{p}(\mf{R})$. This space is given by
\begin{equation*}
	\mathscr{L}^{p}(\mf{R}):=\Big\{(f_{R})_{R\in \mc{R}}\in \prod_{R\in \mc{R}}L^{p}(R^{a},\mu_{R}) \, : \, (f_{R})_{R\in \mc{R}} \text{ is co-Cauchy}\Big\}\Big\slash \sim,
\end{equation*}
where we recall that two sequences $(f_{R})_{R\in \mc{R}}, (g_{R})_{R\in \mc{R}}\in \prod_{R\in \mc{R}}L^{p}(R^{a},\mu_{R})$ are related, $(f_{R})_{R\in \mc{R}}\sim(g_{R})_{R\in \mc{R}}$, if by definition,
\begin{equation*}
	\lim_{R\in \mc{R}} \, \|f_{R}-g_{R}\|_{L^{p}(\mu_{R})}=0.
\end{equation*}
The norm on $\mathscr{L}^{p}(\mf{R})$ is given by
\begin{equation*}
	\|(f_{R})_{R\in \mc{R}}\|_{\mathscr{L}}:=\lim_{R\in \mc{R}} \, \|f_{R}\|_{L^{p}(\mu_{R})}, \quad (f_{R})_{R\in \mc{R}}\in \mathscr{L}^{p}(\mf{R}).
\end{equation*}
In order to define a cocone structure on $\mathscr{L}^{p}(\mf{R})$, we introduce the morphisms
$(\psi_{R})_{R\in \mc{R}}$, $\psi_{R}:L^{p}(R^{a},\mu_{R})\to \mathscr{L}^{p}(\mf{R})$, by
\begin{equation*}
	\psi_{R}(f_{R})=(g_{F})_{F\in \mc{R}}, \quad g_{F}=(\pi^{\ast}_{RF}(f_{R}))_{F\geq R}:=\left\{\begin{array}{ll}
		\pi^{\ast}_{RF}(f_{R}) & \text{ if } R\leq F \\
		0 & \text{ else }
	\end{array}\right.
\end{equation*}
Then it is easy to see that $(\mathscr{L}^{p}(\mf{R}),\psi_{R})$ is a cocone in $\mathfrak{Ban}$. The following result identifies $L^{p}(E^{a},\mu)$ with $\mathscr{L}^{p}(\mf{R})$.
\begin{theorem}
	\label{T4.2}
	Let $E$ be an infinite dimensional real vector space,  $\chi:E\to \C$ a characteristic function and $\mu$ be the corresponding measure on $(E^{a},\mc{B}_{E})$. Then the cocone $(\mathscr{L}^{p}(\mf{R}),\psi_{R})$ defines another realization of the colimit of $(L^{p}(R^{a},\mu_{R}),\pi^{\ast}_{RF})$ and the morphism
	\begin{equation}
		\label{E4}
	\mf{L}_{p}:(\mathscr{L}^{p}(\mf{R}),\psi_{R})\longrightarrow (L^{p}(E^{a},\mu),\varphi_{R}^{\ast}), \quad (f_{R})_{R}\mapsto \lim_{R\in\mc{R}} \, \varphi_{R}^{\ast}(f_{R}),
\end{equation}
defines an isomorphism of cocones with inverse
\begin{equation}
	\label{E4.1}
	\mf{L}_{p}^{-1}:(L^{p}(E^{a},\mu),\varphi_{R}^{\ast})\longrightarrow(\mathscr{L}^{p}(\mf{R}),\psi_{R}), \quad f \mapsto ([\varphi_{R}^{\ast}]^{-1}(\mathbb{E}_{R}(f)))_{R\in \mc{R}},
\end{equation}
where $\mathbb{E}_{R}(f):=\mathbb{E}[f \, | \, \mc{B}^{R}]$ is the conditional expectation of $f\in L^{p}(E^{a},\mu)$ conditioned by the $\sigma$-algebra 
$\mc{B}^{R}:=\{\varphi_{R}^{-1}(B):B\in\mc{B}_{R}\}$.
\end{theorem}

\begin{proof}
	By applying Theorem \ref{T3.3} to the diagram $\mf{R}$, we get that  $(\mathscr{L}^{p}(\mf{R}),\psi_{R})$ defines another realization of the colimit of $(L^{p}(R^{a},\mu_{R}),\pi^{\ast}_{RF})$, where the cocone isomorphism is given by
	\begin{equation}
	\label{E}
	\mf{I}_{p}:(\mathscr{L}^{p}(\mf{R}),\psi_{R})\longrightarrow (L^{p}(\mc{R}_{\infty},\mu_{\infty}),\pi^{\ast}_{R}), \quad (f_{R})_{R\in\mc{R}}\mapsto \lim_{R\in\mc{R}} \, \pi^{\ast}_{R}(f_{R}).
	\end{equation}
     This proves the first part of the theorem. On the other hand, since $I^{\ast}$ is a cocone isomorphism, the composition map $\mf{L}_{p}:=I^{\ast}\circ \mf{I}_{p}$ is also a cocone isomorphism. We illustrate the definition of $\mf{L}_{p}$ in the diagram of Figure \ref{F992}.

     		     	\begin{figure}[h!]
     	\begin{tikzcd}[row sep=large, column sep=large]
     		& (L^{p}(E^{a},\mu), \varphi^{\ast}_{R}) \\
     		(\mathscr{L}^{p}(\mf{R}),\psi_{R})  \arrow{r}{\mf{I}_{p}} \arrow{ru}{\mf{L}_{p}} & (L^{p}(\mc{R}_{\infty},\mu_{\infty}),\pi^{\ast}_{R}) \arrow{u}{I^{\ast}}
     	\end{tikzcd}
        		\caption{Diagram V}
     \label{F992}
 \end{figure}

\noindent To conclude the proof, we have to check that the isomorphism $\mf{L}_{p}$ is given by \eqref{E4}. Let $(f_{R})_{R\in\mc{R}}\in \mathscr{L}^{p}(\mf{R})$. Since $I^{\ast}$ is continuous, we have
 \begin{align*}
 \mf{L}_{p}(f_{R})_{R}&= (I^{\ast}\circ \mf{I}_{p})(f_{R})_{R}
 =\big(\mf{I}_{p}(f_{R})_{R}\big)\circ I
 =\left(\lim_{R\in \mc{R}}\pi^{\ast}_{R}(f_{R})\right)\circ I \\
 &=\lim_{R\in\mc{R}}(f_{R}\circ\pi_{R}\circ I)
 =\lim_{R\in\mc{R}}(f_{R}\circ\varphi_{R})
 =\lim_{R\in\mc{R}}\varphi^{\ast}_{R}(f_{R}).
 \end{align*} All limits are taken in the corresponding $L^{p}$ norms.
This concludes the proof.
\end{proof}

\subsection{Topological vector spaces}\label{S4.1} Let $E$ be a real vector space. If we endow $E$ with a topology so that $E$ is a topological vector space, a natural question arises. Let $\mu$ be the measure on $(E^{a},\mc{B}_{E})$ associated with the characteristic function $\chi:E\to\C$. When does $\mu$ lie on the topological dual space $E'$? A partial answer to this question is contained in the following result whose proof can be found in Yamasaki \cite[Th. 20.1]{Y}.

\begin{theorem}[Yamasaki, \cite{Y}, Th. 20.1]
	Let $E$ be a nuclear topological vector space. If $\chi:E\to\C$ is a continuous characteristic function, the corresponding measure $\mu$ lies on $E'$.
\end{theorem}

Along this section, we will work with a nuclear space $E$ and a given continuous characteristic function $\chi:E\to\C$. Let us denote by $\mc{S}_{E}$ the $\sigma$-algebra defined by
$$\mc{S}_{E}:=\{B\cap E' \, : \, B\in\mc{B}_{E}\} \ (=\mc{B}_{E}\cap E').$$
Then the measure $\mu:\mc{B}_E\to[0,\infty]$ associated with the characteristic function $\chi$ defines a measure on $(E',\mc{S}_{E})$ given by $\mu(B\cap E')=\mu(B)$ for $B\in\mc{B}_{E}$. Subsequently, for each $R\in\mc{R}$, we use also the notation $\varphi_{R}$ to denote the restriction to $E'$ of the morphisms defined in \eqref{Hom3}, that is, $\varphi_{R}:E'\to R^{a}$, $\varphi_{R}(\ell)=\ell|_{R}$ for each $\ell\in E'$. Since $\mu(B)=0$ for every $B\in\mc{B}_{E}$ disjoint from $E'$, the restriction operator
$$T:(L^{p}(E^{a},\mu),\varphi_{R}^{\ast})\longrightarrow (L^{p}(E',\mu),\varphi^{\ast}_{R}), \quad f\mapsto f|_{E'},$$
defines an isomorphism of cocones $\mf{K}^{\ast}(L^{p}\mf{R})$. Substituting directly the space $L^{p}(E^{a},\mu)$ by $L^{p}(E',\mu)$ in the previous section, we obtain the following result.

\begin{theorem}
	Let $E$ be a nuclear topological vector space, $\chi:E\to \C$ be a continuous characteristic function and $\mu$ be the corresponding measure on $(E',\mc{S}_{E})$. Then, for each $1\leq p<\infty$,
	$$\colim\ (L^{p}(R^{a},\mu_{R}),\pi^{\ast}_{RF})= (L^{p}(E',\mu), \varphi^{\ast}_{R}).$$ 
	Moreover, the cocone $(\mathscr{L}^{p}(\mf{R}),\psi_{R})$ defines another realization of the colimit and the morphism
	\begin{equation}
		\label{E4.3}
		\mf{L}_{p}:(\mathscr{L}^{p}(\mf{R}),\psi_{R})\longrightarrow (L^{p}(E',\mu),\varphi_{R}^{\ast}), \quad (f_{R})_{R}\mapsto \lim_{R\in\mc{R}}\varphi^{\ast}_{R}(f_{R}).
	\end{equation}
is a cocone isomorphism with inverse
\begin{equation}
	\label{E4.10}
	\mf{L}_{p}^{-1}:(L^{p}(E',\mu),\varphi_{R}^{\ast})\longrightarrow(\mathscr{L}^{p}(\mf{R}),\psi_{R}), \quad f \mapsto ([\varphi_{R}^{\ast}]^{-1}(\mathbb{E}_{R}(f)))_{R\in \mc{R}},
\end{equation}
where $\mathbb{E}_{R}(f):=\mathbb{E}[f \, | \, \mc{B}^{R}]$ is the conditional expectation of $f\in L^{p}(E',\mu)$ conditioned by the $\sigma$-algebra 
$\mc{B}^{R}:=\{\varphi_{R}^{-1}(B):B\in\mc{B}_{R}\}$.
\end{theorem}

To finalize this section, we proceed to simplify the product space $\mathscr{L}^{p}(\mf{R})$ into a product of a countable number of factor spaces. This will be useful from the point of view of the applications. To this end, we will restrict to an infinite dimensional Fréchet nuclear space $E$. Under this hypothesis, the space $E$ is separable. Let $(\xi_{n})_{n\in\N}\subset E$ be a complete system of $E$, that is, $\text{span}\{\xi_{n}:n\in\N\}$ is dense in $E$. Consider the associated finite dimensional subspaces
$$H_{n}:=\text{span}\{\xi_{1},\cdots,\xi_{n}\}, \quad n\in\N,$$
with the associated directed subset $\mc{H}\equiv(H_{n})_{n\in\N}\subset \mc{R}$. Obviously, the restricted functor $\mf{R}:\mc{H}\to\mf{M}$, subsequently denoted by $\mf{H}$, is a diagram of $\mf{M}$. We denote $\mf{H}\equiv (H_{n}^{a},\mu_{n},\pi_{nm})$, where $\mu_{n}:=\mu_{H_{n}}$ and $\pi_{nm}:=\pi_{H_{n}H_{m}}$ for each $n<m$. A simple computation shows that $(L^{p}(E',\mu),\varphi^{\ast}_{n})$ is a cocone in $\mf{K}^{\ast}(L^{p}\mf{H})$, where for each $n\in\N$, $\varphi_{n}:=\varphi_{H_{n}}: E'\to H^{a}_{n}$.
The product space corresponding to $\mf{H}$ is $\mathscr{L}^{p}(\mf{H})$.	It is convenient to recall that this space is defined by
\begin{equation*}
	\mathscr{L}^{p}(\mf{H}):=\Big\{(f_{n})_{n\in \N}\in \prod_{n\in \N}L^{p}(H_{n}^{a},\mu_{n}) \, : \, (f_{n})_{n\in \N} \text{ is co-Cauchy}\Big\}\Big\slash \sim,
\end{equation*}
where two sequences $(f_{n})_{n\in \N}, (g_{n})_{n\in \N}\in \prod_{n\in \N}L^{p}(H_{n}^{a},\mu_{n})$ are related, $(f_{n})_{n\in \N}\sim(g_{n})_{n\in \N}$, if by definition
\begin{equation*}
	\lim_{n\to\infty}\|f_{n}-g_{n}\|_{L^{p}(\mu_{n})}=0.
\end{equation*}
The norm on $\mathscr{L}^{p}(\mf{H})$ is defined by
\begin{equation*}
	\|(f_{n})_{n\in \N}\|_{\mathscr{L}}:=\lim_{n\to\infty}\|f_{n}\|_{L^{p}(\mu_{n})}, \quad (f_{n})_{n\in \N}\in \mathscr{L}^{p}(\mf{H}).
\end{equation*}
The final result of this section reads as follows.
\begin{theorem}
	\label{T4.5}
	Let $E$ be an infinite dimensional Fréchet nuclear space, $\chi:E\to \C$ a continuous characteristic function and $\mu$ the associated measure on $E'$. Then, for every complete system $(\xi_{n})_{n\in\N}\subset E$, the cocone $(L^{p}(E',\mu),\varphi_{n}^{\ast})$ is the colimit of $(L^{p}(H^{a}_{n},\mu_{n}), \pi^{\ast}_{nm})$. Moreover, the operator
	\begin{equation}
	\label{EqIso}
	\mf{F}_{p}:\mathscr{L}^{p}(\mf{H})\longrightarrow L^{p}(E',\mu), \quad (f_{n})_{n\in\N}\mapsto \lim_{n\to\infty} \varphi^{\ast}_{n}(f_{n}),
	\end{equation}
	defines an isometric isomorphism with inverse
	\begin{equation}
		\label{EqIso2}
		\mf{F}^{-1}_{p}: L^{p}(E',\mu) \longrightarrow \mathscr{L}^{p}(\mf{H}), \quad f \mapsto  ([\varphi_{n}^{\ast}]^{-1}(\mathbb{E}_{n}(f)))_{n\in \N},
	\end{equation}
where $\mathbb{E}_{n}(f):=\mathbb{E}[f  \, | \, \mathcal{B}^{n}]$ is the conditional expectation of $f\in L^{p}(E',\mu)$ conditioned by the $\sigma$-algebra
$$\mc{B}^{n}:=\{\varphi_{n}^{-1}(B) \, : \, B\in \mc{B}_{n}\},$$
where $\mc{B}_{n}$ is the Borel $\sigma$-algebra of $H^{a}_{n}$.
\end{theorem}
\begin{proof}
Let us start by proving the identity of $\sigma$-algebras
\begin{equation}
\label{E4.6}
\sigma(\{\varphi_{n}^{-1}(B) \, : \, n\in \N, \ B\in\mc{B}_{n} \})=\mc{S}_{E},
\end{equation}
where $\mc{B}_{n}$ is the Borel $\sigma$-algebra of $H_{n}^{a}$ and $\varphi_{n}:=\varphi_{H_{n}}:E'\to H^{a}_{n}$ for each $n\in\N$. Note that
$$\mc{S}_{E}=\sigma(\{\xi_{\ast}^{-1}(B) \, : \, \xi\in E, \ B\in \mc{B}_{\R}\}),$$
and also,
$$\sigma(\{\varphi_{n}^{-1}(B) \, : \, n\in \N, \ B\in\mc{B}_{n} \})=\sigma(\{(\xi_{n})_{\ast}^{-1}(B) \, : \, B\in \mc{B}_{\R}, \ n\in\N\}),$$
where for each $\xi\in E$, we are denoting by $\xi_{\ast}$ the functional 
$$\xi_{\ast}:E'\to\R, \quad  \xi_{\ast}(f):=f(\xi).$$ 
Hence, the proof of \eqref{E4.6} reduces to show the following equality of $\sigma$-algebras
\begin{equation*}
	\label{E4.7}
	\mc{S}_{1}:=\sigma(\{(\xi_{n})_{\ast}^{-1}(B) \, : \, B\in \mc{B}_{\R}, \ n\in\N\})=\sigma(\{\xi_{\ast}^{-1}(B) \, : \, \xi\in E, \ B\in \mc{B}_{\R}\})=:\mc{S}_{2}.
\end{equation*}
The inclusion $\mc{S}_{1}\subset\mc{S}_{2}$ is clear. In order to prove $\mc{S}_{2}\subset \mc{S}_{1}$, we have to show that for every $\xi\in E$, the functional $\xi_{\ast}:E^{'}\to\R$ is $\mc{S}_{1}$-measurable. Since $\text{span}\{\xi_{n}:n\in\N\}$ is dense in $E$, there exists a sequence of finite linear combinations $(w_{n})_{n\in\N}\subset E$, 
$$w_{n}:=\sum_{i=1}^{N(n)}\alpha_{i}(n) \xi_{i}, \quad n\in \N,$$
such that $w_{n}\to \xi$ as $n\to\infty$ in the metric topology of $E$ (recall that $E$ is a Fréchet space). Since $(\xi_{n})_{\ast}$ are $\mc{S}_{1}$-measurable and
$$(w_{n})_{\ast}=\sum_{i=1}^{N(n)}\alpha_{i}(n)(\xi_{i})_{\ast}, \quad n\in\N,$$
it follows that $(w_{n})_{\ast}$ are also $\mc{S}_{1}$-measurable. Moreover, $(w_{n})_{\ast}\to \xi_{\ast}$ pointwise as $n\to\infty$ since for each $f\in E'$,
$$\lim_{n\to\infty}(w_{n})_{\ast}(f)=\lim_{n\to\infty}f(w_{n})=f(\xi)=\xi_{\ast}(f).$$
Hence, $\xi_{\ast}$ is $\mc{S}_{1}$-measurable as is the pointwise limit of $\mc{S}_{1}$-measurable functions. This shows that $\mc{S}_{1}=\mc{S}_{2}$ and completes the proof of \eqref{E4.6}. Arguing as in the proof of Theorem \ref{T3.1}, it follows that the space 
$$\bigcup_{n\in\N}\varphi^{\ast}_{n}(L^{p}(H^{a}_{n},\mu_{n}))\subset L^{p}(E',\mu) ,$$
is dense in $L^{p}(E',\mu)$. By Lemma \ref{L2.1}, this implies that $(L^{p}(E',\mu),\varphi_{n}^{\ast})$ is the colimit of $(L^{p}(H^{a}_{n},\mu_{n}), \pi^{\ast}_{nm})$. The proof of the identification $L^{p}(E',\mu)\simeq \mathscr{L}^{p}(\mf{H})$ via the isomorphism \eqref{EqIso}--\eqref{EqIso2}, follows the same steeps of the proof of Theorem \ref{T3.2}.
\end{proof}

\subsection{Gaussian measures}\label{S4.2} In this subsection we focus on the particular case of Gaussian measures. Let $E$ be an infinite dimensional Fréchet nuclear space. A measure $\mu$ on $(E',\mc{S}_{E})$ is called \textit{Gaussian} if there exists an inner product $\mf{C}(\cdot,\cdot)$ on $E$ such that 
\begin{equation}
	\label{chi}
\chi(\xi)=e^{-\frac{1}{2}\mf{C}(\xi,\xi)}, \quad \xi\in E,
\end{equation}
where $\chi: E\to\C$ is the characteristic function of $\mu$. The inner product $\mf{C}$ is usually called the covariance operator. Let $R$ be a finite dimensional subspace of $E$ and let $\mf{C}_{R}$ denote the restriction of $\mf{C}$ to $R$, i.e., $\mf{C}_{R}:R\times R\to \R$. Then, if $\{\xi_{1},\xi_{2},\cdots,\xi_{n}\}$ is a basis of $R$, $\mf{C}_{R}$ can be represented as a symmetric matrix by
$$\mf{C}_{R}\equiv \left(\mf{C}_{ij}\right)_{1\leq i,j \leq n}, \quad \mf{C}_{ij}=\mf{C}(\xi_{i},\xi_{j}).$$
We denote its inverse by $\mf{C}^{-1}_{R}\equiv (\mf{C}^{-1}_{ij})_{1\leq i,j \leq n}$.
By Bochner theorem and \eqref{chi}, it is easy to see that the measure $\mu_{R}$ is explicitly given by 
$$\mathrm{d}\mu_{R}(\mathbf{x})=\frac{1}{\sqrt{(2\pi)^{n}\det[\mf{C}_{R}]}}\exp\left\{-\frac{1}{2}\sum_{i,j=1}^{n}\mf{C}^{-1}_{ij}x_{i}x_{j}\right\} \, \mathrm{d}^{n}\mathbf{x},$$
$$\mathbf{x}:=(x_{1},\cdots,x_{n})\in R^{a}\simeq \R^{n},$$
where ${\rm{d}}^{n}\mathbf{x}$ is the Lebesgue measure on $R^{a}$. Keeping the notation of Section \ref{S4.1}, we obtain the following result as a direct application of Theorem \ref{T4.5} to this particular case. It establishes a small extension of known results for the $L^{p}$-case with $p\neq 2$.
\begin{theorem}
	\label{T4.5.2}
	Let $E$ be an infinite dimensional Fréchet nuclear space and $\mu$ a Gaussian measure on $E'$ with covariance operator $\mf{C}$. Then for every complete system $(\xi_{n})_{n\in\N}\subset E$, the operator
	$$\mf{F}_{p}:\mathscr{L}^{p}(\mf{H})\longrightarrow L^{p}(E',\mu), \quad (f_{n})_{n\in\N}\mapsto \lim_{n\to\infty} \varphi^{\ast}_{n}(f_{n}),$$
	defines an isometric isomorphism with inverse
	\begin{equation}
		\label{eq3.3.3}
		\mf{F}^{-1}_{p}: L^{p}(E',\mu) \longrightarrow \mathscr{L}^{p}(\mf{H}), \quad f \mapsto  ([\varphi_{n}^{\ast}]^{-1}(\mathbb{E}_{n}(f)))_{n\in \N}.
	\end{equation}
For the following formula, suppose that $(\xi_{n})_{n\in\N}$ is linearly independent and put $R=H_{n}$. Consequently, for each $f\in L^{1}(E',\mu)$, there exists a sequence $(f_{n})_{n\in\N}\in \prod_{n\in \N}L^{1}(H_{n}^{a},\mu_{n})$ such that
\begin{equation}
	\int_{E'}f \, \mathrm{d}\mu=\lim_{n\to\infty} \int_{H_{n}^{a}} f_{n}(\mathbf{x})\cdot  \frac{\exp\left\{-\frac{1}{2}\sum_{i,j=1}^{n}\mf{C}^{-1}_{ij}x_{i}x_{j}\right\}}{\sqrt{(2\pi)^{n}\det[\mf{C}_{R}]}} \, \mathrm{d}^{n}\mathbf{x}.
\end{equation}
Moreover, the sequence $(f_{n})_{n\in\N}$ can be chosen to be $f_{n}=[\varphi_{n}^{\ast}]^{-1}(\mathbb{E}_{n}(f))$ for each $n\in \N$. 
\end{theorem}

\section{Application I: Constructive Quantum Field Theory}\label{S4.3}

In this section, we are mainly interested in the application of the preceding abstract results to constructive (Euclidean) Quantum Field Theory (QFT) via the Osterwalder--Schrader axioms as stated in Glimm--Jaffe \cite[Ch. 6]{GJ}. Roughly speaking, the Osterwalder--Schrader axioms are stated in terms of a Borel probability measure $\mu$ on the space of tempered distributions $\mathscr{S}'(\R^{d})$, where $d$ is the space-time dimension, satisfying a series of axioms \cite[pp. 91–92]{GJ}. The measure $\mu$ plays the same role as does the Feynman--Kac measure in quantum mechanics. One of the major concerns of QFT is the computation of the Schwinger (or correlation) functions that are given by 
$$\mf{S}_{k}:\prod_{i=1}^{k}\mathscr{S}(\R^{d})\to \R, \quad \mf{S}_{k}(f_{1},\cdots,f_{k}):=\int_{\mathscr{S}'(\R^{d})}\phi(f_{1})\cdots\phi(f_{k}) \, \mathrm{d}\mu(\phi),$$
for positive integers $k\geq 1$.
We proceed to describe the measure $\mu$ in the case of the scalar free field with mass $m\in(0,\infty)$ and spin zero. It is well known that the space of Schwartz functions $E=\mathscr{S}(\R^{d})$ is an infinite dimensional Fréchet nuclear space. Given $m\in(0,\infty)$, let $\mu_{\mf{C}}$ be the Gaussian measure on $\mathscr{S}'(\R^{d})$ with covariance operator
$$\mf{C}:\mathscr{S}(\R^{d})\times \mathscr{S}(\R^{d})\longrightarrow \R, \quad  \mf{C}(f,g):=\int_{\R^{d}}f(x)(-\Delta+m^{2})^{-1}g(x) \, \mathrm{d}x.$$
The measure space $(\mathscr{S}'(\R^{d}),\mu_{\mf{C}})$ is known as the path space for free particles with mass $m>0$. For the sake of clarity, let us fix $m=1$ throughout the rest of this section. Under this assumption, the completion of the pre-Hilbert space $(\mathscr{S}(\R^{d}), \mf{C})$ leads to the Sobolev space $H^{-1}(\R^{d})$. The computation of the Schwinger functions in this case is easy as a consequence of the Bochner theorem. Suppose first that $f_{1},\cdots,f_{k}$ are linearly independent. Then
\begin{align*}
\mf{S}_{k}(f_{1},\cdots,f_{k})&=\int_{\mathscr{S}'(\R^{d})}\phi(f_{1})\cdots\phi(f_{k}) \, \mathrm{d}\mu_{\mf{C}}(\phi)\\
&=\int_{R^{a}}x_{1}\cdots x_{k}\frac{\exp\left\{-\frac{1}{2}\sum_{i,j=1}^{k}\mf{C}^{-1}_{ij}x_{i}x_{j}\right\}}{\sqrt{(2\pi)^{k}\det[\mf{C}_{R}]}} \, \mathrm{d}^{k}\mathbf{x},
\end{align*}
where $R:=\text{span}\{f_{1},\cdots,f_{k}\}\subset \mathscr{S}(\R^{d})$ and therefore
$$\mf{C}_{R}\equiv \left(\mf{C}_{ij}\right)_{1\leq i,j \leq k}, \quad \mf{C}_{ij}=\mf{C}(f_{i},f_{j}).$$
By the Wick--Isserlis formula the above integral is zero if $k$ odd while
$$\mf{S}_{k}(f_{1},\cdots,f_{k})=\frac{1}{2^{k/2}(k/2)!}\sum_{\sigma\in\Sigma_{k}}\prod_{j=1}^{k/2}\mf{C}(f_{\sigma(2j-1)},f_{\sigma(2j)}),$$
if $k$ is even. By continuity, these Wick--Isserlis identities also hold for linearly dependent test functions. Here and in the sequel, the notation $\Sigma_{k}$ stands for the symmetric group of degree $k$. 
\par In the forthcoming subsection we proceed to analyze the Hilbert space $L^{2}(\mathscr{S}'(\R^{d}),\mu_{\mf{C}})$ under the perspective of the results of this article.

\subsection{Analysis of the free field} Let $\mc{R}$ be the class of finite dimensional subspaces of $\mathscr{S}(\R^{d})$ and consider the diagram $ (R^{a},\mu_{R},\pi_{RF})$ of $\mf{M}$ represented by the functor $\mf{R}:\mc{R}\to\mf{M}$. Recall that for each $R,F\in\mc{R}$ with $R\leq F$, $\pi_{RF}$ are the measurable functions defined by
$$
\pi_{RF}:(F^{a},\mu_{F})\longrightarrow (R^{a},\mu_{R}), \quad \ell_{F}\mapsto \ell_{F}|_{R}, \quad R\leq F.$$
Recall also that for each $R\in\mc{R}$, the map $\varphi_{R}$ is given through
$$\varphi_{R}: \mathscr{S}'(\R^{d})\longrightarrow R^{a}, \quad \varphi_{R}(\phi)=\phi|_{R}.$$
In this case, the space $\mathscr{L}^{2}(\mf{R})$ is given explicitly by
\begin{equation*}
	\mathscr{L}^{2}(\mf{R}):=\Big\{(f_{R})_{R\in \mc{R}}\in \prod_{R\in \mc{R}}L^{2}(R^{a},\mu_{R}) \, : \, (f_{R})_{R\in \mc{R}} \text{ is co-Cauchy}\Big\}\Big\slash \sim,
\end{equation*}
where two sequences $(f_{R})_{R\in \mc{R}}, (g_{R})_{R\in \mc{R}}\in \prod_{R\in \mc{R}}L^{2}(R^{a},\mu_{R})$ are related, $(f_{R})_{R\in \mc{R}}\sim(g_{R})_{R\in \mc{R}}$, if by definition,
\begin{equation*}
	\lim_{R\in \mc{R}} \, \|f_{R}-g_{R}\|_{L^{2}(\mu_{R})}=0.
\end{equation*}
By Theorem \ref{T4.2}, the space $\mathscr{L}^{2}(\mf{R})$ is isometrically isomorphic to $L^{2}(\mathscr{S}'(\R^{d}),\mu_{\mf{C}})$ via the isometry
\begin{equation*}
	\mf{L}_{2}: \mathscr{L}^{2}(\mf{R}) \longrightarrow L^{2}(\mathscr{S}'(\R^{d}),\mu_{\mf{C}}), \quad (f_{R})_{R}\mapsto \lim_{R\in\mc{R}}\varphi^{\ast}_{R}(f_{R}).
\end{equation*}
with inverse
\begin{equation}
	\label{EL21}
	\mf{L}^{-1}_{2}: L^{2}(\mathscr{S}'(\R^{d}),\mu_{\mf{C}}) \longrightarrow \mathscr{L}^{2}(\mf{R}), \quad f\mapsto ([\varphi_{R}^{\ast}]^{-1}(\mathbb{E}_{R}(f)))_{R\in\mc{R}},
\end{equation}
where $\mathbb{E}_{R}(f):=\mathbb{E}[f \, | \, \mc{B}^{R}]$ is the conditional expectation of $f$ conditioned by the $\sigma$-algebra 
$$\mc{B}^{R}:=\{\varphi_{R}^{-1}(B) \, : \, B\in \mc{B}_{R}\}.$$
The main result of this subsection expresses the inverse operator $\mf{L}_{2}^{-1}$ in terms of orthogonal projections onto some subspaces of $L^{2}(\mathscr{S}'(\R^{d}),\mu_{\mf{C}})$. 
\par Let us introduce some results and notation in order to reach our goal. Recall that given $k_{1},\cdots, k_{n}\in \N\cup\{0\}$ and real-valued random variables $X_{1},\cdots,X_{n}$ with finite moments, the \textit{Wick product} $\normord{X_{1}^{k_{1}}\cdots X_{n}^{k_{n}}}$ is a polynomial in $X_{1},\cdots, X_{n}$ of total degree $k_{1}+\cdots +k_{n}$ defined recursively as follows:
\begin{equation*}
	\left\{
	\begin{array}{l}
		\normord{X_{1}^{0}\cdots X_{n}^{0}} \, =1, \\
		\frac{\partial}{\partial X_{i}}\normord{X_{1}^{k_{1}}\cdots X_{n}^{k_{n}}} \, =k_{i} \normord{X_{1}^{k_{1}}\cdots X_{i}^{k_{i}-1}\cdots X_{n}^{k_{n}}}, \quad \text{for each } 1\leq i \leq n \text{ with } k_{i}\geq 1, \\
		\mathbb{E}(\normord{X_{1}^{k_{1}}\cdots X_{n}^{k_{n}}})=0, \quad \text{ for not all } k_{i}=0.
	\end{array}\right.
\end{equation*}
We refer the reader to Simon \cite[pp. 9--12]{Si} for properties and consequences of the definition of the Wick product. We denote by $\Gamma_{n}$ the closure of the subspaces
$$\spann\{\normord{\phi(f_{1})\cdots \phi(f_{n})}\  : \, f_{1},\cdots, f_{n}\in\mathscr{S}(\R^{d})\}\subset L^{2}(\mathscr{S}'(\R^{d}),\mu_{\mf{C}}).$$
Due to the properties of the Wick product, it is easy to see that $\Gamma_{n}\perp \Gamma_{m}$ for each $n\neq m$. The Wiener--Itô--Segal decomposition (see for instance Obata \cite[Th. 2.3.3]{O}) establishes that 
\begin{equation*}
	L^{2}(\mathscr{S}'(\R^{d}),\mu_{\mf{C}})=\bigoplus_{n\in\N\cup\{0\}} \Gamma_{n},
\end{equation*}
where $\Gamma_{0}$ is the space of constants and the right hand side means the completed orthogonal direct sum.
Thanks to this decomposition it can be proved that for each contraction $T: H^{-1}(\R^{d})\to H^{-1}(\R^{d})$, there exists a unique (lifted) contraction $\Gamma[T]:	L^{2}(\mathscr{S}'(\R^{d}),\mu_{\mf{C}})\to	L^{2}(\mathscr{S}'(\R^{d}),\mu_{\mf{C}})$ satisfying
\begin{equation*}
	\Gamma[T](\normord{\phi(f_{1})\cdots \phi(f_{n})})= \, \normord{\phi(T(f_{1}))\cdots \phi(T(f_{n}))}
\end{equation*}
for each $f_{1},\cdots, f_{n}\in\mathscr{S}(\R^{d})$, with $\Gamma[T]1=1$ and $\phi$ extended to $H^{-1}$ by $L^{2}$ continuity. For a proof of these facts we refer the reader to Simon \cite[\S 1.4]{Si}. Given a finite dimensional subspace $R\in\mc{R}$, we can regard $R$ as a subspace of $H^{-1}(\R^{d})$ under the inclusion $\mathscr{S}(\R^{d})\hookrightarrow H^{-1}(\R^{d})$. Theorem III.8 of Simon \cite{Si} states that
\begin{equation}
	\label{EL211}
\mathbb{E}_{R}(f)=\Gamma[P_{R}](f),
\end{equation}
for each $f\in L^{2}(\mathscr{S}'(\R^{d}),\mu_{\mf{C}})$, where $P_{R}:H^{-1}(\R^{d})\to H^{-1}(\R^{d})$ the orthogonal projection onto $R$.  Then, comparing \eqref{EL21} and \eqref{EL211}, we obtain the main result of this subsection.
\begin{theorem}
	The isomorphism $\mf{L}^{-1}_{2}: L^{2}(\mathscr{S}'(\R^{d}),\mu_{\mf{C}}) \longrightarrow \mathscr{L}^{2}(\mf{R})$ can be alternatively described as
	$$\mf{L}^{-1}_{2}(f)=([\varphi_{R}^{\ast}]^{-1}(\Gamma[P_{R}](f)))_{R\in\mc{R}}, \quad f\in L^{2}(\mathscr{S}'(\R^{d}),\mu_{\mf{C}}).$$
	Therefore, for each $f\in L^{2}(\mathscr{S}'(\R^{d}),\mu_{\mf{C}})$, the following identity holds:
	\begin{equation*}
		\int_{\mathscr{S}'(\R^{d})} f(\phi) \, \mathrm{d}\mu_{\mf{C}}(\phi)=\lim_{R\in \mc{R}} \, \int_{R^{a}} [\varphi^{\ast}_{R}]^{-1}(\Gamma[P_{R}](f)) \, \mathrm{d}\mu_{R}.
	\end{equation*}
\end{theorem}
\subsection{Analysis of interacting fields} The main interest in QFT lies in interacting fields. The program of constructing the path measure $\mu$ increases in difficulty with the dimension $d$ of space-time. Here we focus on small polynomial non-linearities in the space-time dimension $d=2$, commonly known as $\mc{P}(\phi)_{2}$-theories. The construction of the measure for this theories was completed by Nelson \cite{Ne}, Guerra--Rosen--Simon \cite{GRS}, Glim--Jaffe \cite{GJ0} and Dimock--Glimm \cite{DG}. In finite volume (cutoff) $\mc{P}(\phi)_{2}$-theories, the measure $\mu$ can be expressed as 
$$\mathrm{d}\mu=Z^{-1}e^{-V} \, \mathrm{d}\mu_{\mf{C}},\qquad Z=\int e^{-V}\,\mathrm{d}\mu_{\mf{C}},$$
where $\mu_{\mf{C}}$ is the Gaussian measure on $\mathscr{S}'(\R^{2})$ with covariance $\mf{C}$ and 
\begin{equation}
	\label{Lp}
e^{-V}\in \bigcap_{p<\infty} L^{p}(\mathscr{S}'(\R^{2}),\mu_{\mf{C}}).
\end{equation}
Let us recall the definition of $V\in L^{2}(\mathscr{S}'(\R^{2}),\mu_{\mf{C}})$ for the $\phi^{k}_{2}$-theory, with even $k\geq2$ and nonnegative compactly supported cutoff function $g\in L^{p}(\R^{2})$, $p>1$. Consider a compactly supported smooth function $h\in \mc{C}^{\infty}_{0}(\R^{2})$ with $h\geq 0$ and $\int h(x) \, \mathrm{d}x=1$. Define the approximate $\delta$-function centered at $x\in \R^{2}$ by
$$\delta_{n,x}(y):=n^{2} \, h(n(x-y)), \quad y\in \R^{2}, \ n\in\N.$$
Note that $\delta_{n,x}\in\mathscr{S}(\R^{2})$ and $\delta_{n,x}\to \delta_{x}$ as $n\to \infty$ in $\mathscr{S}'(\R^{2})$, where $\delta_{x}$ is the Dirac delta function centered at $x$. We define $V$ as the limit
\begin{equation}
	\label{Lim}
	V(\phi):=\lim_{n\to\infty}\int_{\R^{2}} \normord{\phi^{k}(\delta_{n,x})}\, g(x) \, \mathrm{d}x \quad \text{ in } L^{2}(\mathscr{S}'(\R^{2}),\mu_{\mf{C}}),
\end{equation}
where $\normord{\phi^{k}}$ is the Wick power given explicitly by
\begin{equation*}
	\normord{\phi^{k}(f)} \, = \sum_{j=0}^{[k/2]}\frac{(-1)^{j} \, k!}{(k-2j)! \, j! \, 2^{j}}\, \mf{C}(f,f)^{j}\, \phi(f)^{k-2j}, \quad f\in\mathscr{S}(\R^{2}).
\end{equation*}
The convergence of the limit \eqref{Lim} follows from Simon \cite[Th. V.3]{Si} or Glim--Jaffe \cite[Pr. 8.5.1]{GJ}. Moreover, Theorem 8.5.3 of \cite{GJ} shows that the limit \eqref{Lim} also holds in $L^{p}(\mu_{\mf{C}})$ for each $p\geq 1$. Under this definition of $V$, property \eqref{Lp} is a consequence of \cite[Th. V.7]{Si} or \cite[Th. 8.6.2]{GJ}. Moreover, 
\begin{equation}
	\label{LimE}
	e^{-V}=\lim_{n\to\infty}\exp\left(-\int_{\R^{2}} \normord{\phi^{k}(\delta_{n,x})} \, g(x) \, \mathrm{d}x\right) \quad \text{ in } L^{p}(\mathscr{S}'(\R^{2}),\mu_{\mf{C}}), \, p\geq 1.
\end{equation}
Subsequently, for the sake of notation, we will denote 
$$\normord{\phi^{k}_{n}(g)} \, = \int_{\R^{2}} \normord{\phi^{k}(\delta_{n,x})} \, g(x) \, \mathrm{d}x,$$
for each compactly supported $g\in L^{p}(\R^{2})$.
\par We proceed to describe the functions $\normord{\phi^{k}_{n}(g)}$ and $\exp(-\normord{\phi^{k}_{n}(g)} )$ under the categorical schemes of this article. Suppose from now on that $g$ is Riemann integrable. Fix $n\in \N$ and consider the family of finite dimensional subspaces
$$H_{m,n}:=\text{span}\left\{\delta_{n,x} \, : \, x\in \tfrac{1}{2^{m}}\mathbb{Z}^{2}, \, |x|\leq m\right\}\subset \mathscr{S}(\R^{2}), \quad m\in \N.$$
Subsequently, we will denote $\mf{N}_{m}:=\{x\in\tfrac{1}{2^{m}}\mathbb{Z}^{2} \, : \, |x|\leq m  \}$. Note that $(H_{m,n})_{m\in\N}$ forms a filtered family of subspaces of $\mathscr{S}(\R^{2})$:
\begin{equation*}
	H_{1,n}\leq H_{2,n} \leq \cdots \leq H_{m,n} \leq \cdots
\end{equation*}
Suppose that the function $h$ has a nowhere vanishing Fourier transform, that is, $\hat{h}(\xi)\neq 0$ for every $\xi\in \R^{2}$. Then, making use of standard Fourier transform methods for the space of tempered distributions, it can be proved that the subspace 
$$\text{span}\{\delta_{n,x} \, : \, x\in \mathbb{Z}[1/2]^{2}\}\subset \mathscr{S}(\R^{2}),$$
is dense in $\mathscr{S}(\R^{2})$. Hence 
$\bigcup_{m\in\N} H_{m,n}$
is dense in $\mathscr{S}(\R^{2})$. Applying the techniques of Theorem \ref{T4.5.2}, we obtain that for each $p\geq 1$, the operator 
\begin{equation}
	\label{ISOM}
	\mf{F}_{p,n}:\mathscr{L}^{p}(\mf{H})\longrightarrow L^{p}(\mathscr{S}'(\R^{2}),\mu_{\mf{C}}), \quad (f_{m})_{m\in\N}\mapsto \lim_{m\to\infty} \varphi^{\ast}_{m,n}(f_{m}),
\end{equation}
is an isometric isomorphism,
where $\mf{H}:\N\to \mf{M}$ is the diagram $\mf{H}\equiv (H_{m,n}^{a},\mu_{m,n},\pi_{m m', n})$ with 
$$\mu_{m,n}:=\mu_{H_{m,n}} \quad \text{and} \quad \pi_{mm', n}:=\pi_{H_{m,n}H_{m',n}}.$$ 
In \eqref{ISOM}, for each $m\in \N$, $\varphi^{\ast}_{m,n}$ stands for the application of the $L^{p}$-functor to
\begin{equation}
	\label{defp}
	\varphi_{m,n}:\mathscr{S}'(\R^{2}) \longrightarrow H^{a}_{m,n}, \quad \varphi_{m,n}(\phi):=\phi|_{H_{m,n}}=\sum_{\ell\in \mf{N}_{m}} \phi(\delta_{n,\ell}) \, \delta_{n,\ell}^{a}.
\end{equation}
The next result gives the sequence $(f_{m})_{m\in \N}\in\mathscr{L}^{p}(\mf{H})$ that corresponds with the $L^{p}$-functions $\normord{\phi^{k}_{n}(g)}$ and $\exp(-\normord{\phi^{k}_{n}(g)})$ via the isometry $\mf{F}_{p,n}$.
\begin{theorem}
	\label{Th6.2}
	Let $n\in\N$, $p\geq 1$ and suppose that $g$ is Riemann integrable. Let $(f_{m})_{m\in\N}\in \mathscr{L}^{p}(\mf{H})$ be given by
	\begin{equation*}
		f_{m}(\mathbf{z})=\frac{1}{2^{2m}}\sum_{\ell\in \mf{N}_{m}} \normord{z^{k}_{\ell}} \, g(\ell), \quad \mathbf{z}:=(z_{\ell})_{\ell\in \mf{N}_{m}}\equiv \sum_{\ell\in\mf{N}_{m}}z_{\ell}\, \delta_{n,\ell}^{a}\in H_{m,n}^{a}.
	\end{equation*}
	Then, the following identities hold:
	\begin{equation}
		\label{Ef2}
		\mf{F}_{p,n}(f_{m})_{m\in\N}=\int_{\R^{2}} \normord{\phi^{k}(\delta_{n,x})} \, g(x) \, \mathrm{d}x.
	\end{equation}
		\begin{equation}
	\label{Efe}
	\mf{F}_{p,n}(\exp(-f_{m}))_{m\in\N}=\exp\left(-\int_{\R^{2}} \normord{\phi^{k}(\delta_{n,x})} \, g(x) \, \mathrm{d}x\right).
\end{equation}
\end{theorem}
\begin{proof}
In the following lines, we will omit the subscript $n$ in the notations $H_{m,n}$, $\mu_{m,n}$, $\pi_{mm',n}$ and $\varphi_{m,n}$ as it is fixed throughout the proof.
Under the hypothesis on $g$, Proposition 8.7.1 of \cite{GJ} establishes that
\begin{equation}
	\label{E5.9}
	\lim_{m\to\infty}\frac{1}{2^{2m}}\sum_{\ell\in \mf{N}_{m}} \normord{\phi^{k}(\delta_{n,\ell})} \, g(\ell)=\int_{\R^{2}}\normord{\phi^{k}(\delta_{n,x})} \, g(x) \, \mathrm{d}x,
\end{equation}
in $L^{p}(\mathscr{S}'(\R^{2}),\mu_{\mf{C}})$. We proceed to prove the first identity \eqref{Ef2}. We start by proving that the sequence $(f_{m})_{m\in \N}$ is co-Cauchy.  It is convenient to recall that for each $m_{1},m_{2}\in \N$, $m_{1}<m_{2}$, the map $\pi_{m_{1}m_{2}}:H^{a}_{m_{2}}\to H^{a}_{m_{1}}$ is given by
$$\pi_{m_{1}m_{2}}:H^{a}_{m_{2}}\longrightarrow H^{a}_{m_{1}}, \quad \pi_{m_{1}m_{2}}(\ell)=\ell\,|_{H_{m_{1}}}, \quad \ell\in H^{a}_{m_{2}}.$$
Then, a simple computation gives
\begin{align*}
	\|f_{m_{2}}-\pi^{\ast}_{m_1 m_2}(f_{m_1})\|_{L^{p}(\mu_{m_{2}})}^{p}
	& =\int_{H_{m_{2}}^{a}}\Big| \frac{1}{2^{2m_{2}}}\sum_{\ell\in\mf{N}_{m_{2}}} \normord{z_{\ell}^{k}} g(\ell) - \frac{1}{2^{2m_{1}}}\sum_{\ell\in\mf{N}_{m_{1}}} \normord{z_{\ell}^{k}} g(\ell)\Big|^{p} \, \mathrm{d}\mu_{m_{2}}(\mathbf{z}) \\
	& =\int_{\mathscr{S}'} \Big| \frac{1}{2^{2m_{2}}}\sum_{\ell\in\mf{N}_{m_{2}}} \normord{ \phi^{k}(\delta_{n,\ell})} g(\ell) - \frac{1}{2^{2m_{1}}}\sum_{\ell\in\mf{N}_{m_{1}}}\normord{\phi^{k}(\delta_{n,\ell})} g(\ell)\Big|^{p} \, \mathrm{d}\mu_{\mf{C}} \\
	&=\Big\|\frac{1}{2^{2m_{2}}}\sum_{\ell\in\mf{N}_{m_{2}}} \normord{\phi^{k}(\delta_{n,\ell})} g(\ell) - \frac{1}{2^{2m_{1}}}\sum_{\ell\in\mf{N}_{m_{1}}} \normord{\phi^{k}(\delta_{n,\ell})} g(\ell)\Big\|_{L^{p}(\mu_{\mf{C}})}^{p}.
\end{align*}
Therefore, it follows from the $L^{p}$-convergence of the limit \eqref{E5.9} that $(f_{m})_{m\in\N}$ is co-Cauchy and $(f_{m})_{m\in\N}\in \mathscr{L}^{p}(\mf{H})$.
On the other hand, recalling the definition of $\varphi_{m}$, equation \eqref{defp}, we deduce from \eqref{E5.9} that
\begin{align*}
	\lim_{m\to\infty}\|\normord{\phi^{k}_{n}(g)} - \, \varphi_{m}^{\ast}(f_{m})\|_{L^{p}(\mu_{\mf{C}})}=\lim_{m\to\infty}\Big\|\normord{\phi^{k}_{n}(g)}- \, \frac{1}{2^{2m}}\sum_{\ell\in \mf{N}_{m}}\normord{\phi^{k}(\delta_{n,\ell})} g(\ell) \Big\|_{L^{p}(\mu_{\mf{C}})}=0.
\end{align*}
This concludes the proof of equation \eqref{Ef2}. For the second identity \eqref{Efe}, we follow the scheme of the proof of Proposition 8.7.2 in \cite{GJ}. We introduce the function
$$\mf{P}_{m}(t):=-t \, \varphi_{m}^{\ast}(f_{m})-(1-t) \normord{\phi^{k}_{n}(g)} \, , \quad t\in[0,1], \, m\in \N.$$
Then, applying H\"{o}lder's inequality, we deduce that
\begin{align*}
	& \|\exp(-\normord{\phi^{k}_{n}(g)} )-\varphi_{m}^{\ast}(\exp(-f_{m}))\|_{L^{p}(\mu_{\mf{C}})}=\|\exp(\mf{P}_{m}(0))-\exp(\mf{P}_{m}(1))\|_{L^{p}(\mu_{\mf{C}})}\\
	&\leq \int_{0}^{1}\Big\| \frac{d}{dt}\exp(\mf{P}_{m}(t)) \, \mathrm{d}t \Big\|_{L^{p}(\mu_{\mf{C}})}\leq \sup_{t\in[0,1]} \|(\normord{\phi^{k}_{n}(g)} - \, \varphi_{m}^{\ast}(f_{m})) \exp(\mf{P}_{m}(t))\|_{L^{p}(\mu_{\mf{C}})} \\
	&\leq \|\normord{\phi^{k}_{n}(g)}- \, \varphi_{m}^{\ast}(f_{m})\|_{L^{2p}(\mu_{\mf{C}})} \, \sup_{t\in[0,1]} \|\exp(\mf{P}_{m}(t))\|_{L^{2p}(\mu_{\mf{C}})}.
\end{align*}
The first factor converges to zero as $m\to\infty$ by the arguments used to establish identity \eqref{E5.9} and the second factor is uniformly bounded. For fixed $n$, the even Wick polynomial is bounded below, independently of the translation parameter. Since $g\geq0$ and $2^{-2m}\sum_{\ell\in\mf{N}_{m}}g(\ell)$ is bounded, both potentials are bounded below by $-B_{n}$, and consequently $\exp(\mf{P}_{m}(t))\leq e^{B_{n}}$. This proves identity \eqref{Efe} and concludes the proof.
\end{proof}
A direct consequence of Theorem \ref{Th6.2} is the following integral formulae.
\begin{corollary}
	\label{CC}
	Let $n\in\N$. Then, the following identities hold:
	\begin{equation*}
		\int_{\mathscr{S}'(\R^{2})} \exp(-\normord{\phi^{k}_{n}(g)}) \, \mathrm{d}\mu_{\mf{C}}=\lim_{m\to\infty} \int_{H_{m,n}^{a}}\exp\left(-\frac{1}{2^{2m}}\sum_{\ell\in \mf{N}_{m}}\normord{z^{k}_{\ell}} g(\ell)\right) \, \mathrm{d}\mu_{m,n}(\mathbf{z}),
	\end{equation*}
\begin{equation*}
	\int_{\mathscr{S}'(\R^{2})} e^{-V} \, \mathrm{d}\mu_{\mf{C}}=\lim_{n\to\infty}\lim_{m\to\infty} \int_{H_{m,n}^{a}}\exp\left(-\frac{1}{2^{2m}}\sum_{\ell\in \mf{N}_{m}} \normord{z^{k}_{\ell}} g(\ell)\right) \, \mathrm{d}\mu_{m,n}(\mathbf{z}).
\end{equation*}
\end{corollary}
\noindent In order to give the complete formulas in Corollary \ref{CC}, we recall that the measure $\mu_{m,n}$ can be explicitly given by 
$$\mathrm{d}\mu_{m,n}(\mathbf{z})=\frac{\exp\left\{-\frac{1}{2}\sum_{i,j\in\mf{N}_{m}}\mf{C}^{-1}_{ij,n}z_{i}z_{j}\right\}}{\sqrt{(2\pi)^{|\mf{N}_{m}|}\det[\mf{C}_{m,n}]}} \, \mathrm{d}^{|\mf{N}_{m}|}\mathbf{z},$$
where we have denoted 
\begin{equation*}
	\mf{C}_{m,n}\equiv (\mf{C}(\delta_{n,i},\delta_{n,j}))_{i,j\in \mf{N}_{m}}, \quad \mf{C}_{m,n}^{-1}\equiv (\mf{C}^{-1}_{ij,n})_{i,j\in\mf{N}_{m}},
\end{equation*}
where $\mf{C}_{m,n}^{-1}$ is the inverse matrix of $\mf{C}_{m,n}$. 
\par Subsequently, to emphasize the dependence of $n$ in $V$, let us denote by $V_{n}:\mathscr{S}'(\R^{2})\to \R$ the function defined through $V_{n}(\phi):= \, \normord{\phi^{k}_{n}(g)}$ and also the functions
$$V_{n,m}: H_{m,n}^{a}\longrightarrow \R, \quad V_{n,m}(\mathbf{z}):=\frac{1}{2^{2m}}\sum_{\ell\in \mf{N}_{m}} \normord{z^{k}_{\ell}} g(\ell).$$
Moreover, for each $f\in L^{p}(\mathscr{S}'(\R^{2}),\mu_{\mf{C}})$, the notation $\mathbb{E}_{m,n}(f):=\mathbb{E}[f \, | \, \mc{B}^{n,m}]$ stands for the conditional expectation of $f$ conditioned by the $\sigma$-algebra $\mc{B}^{n,m}:=\{\varphi_{m,n}^{-1}(B) \, : \, B\in \mc{B}_{m,n}\}$, where $\mc{B}_{m,n}$ is the Borel $\sigma$-algebra of $H^{a}_{m,n}$. Under this notation, we can state another consequence of Theorem \ref{Th6.2}.
\begin{theorem}
	\label{th5.3}
	Let $n\in \N$ and $f\in L^{q}(\mathscr{S}'(\R^{2}),\mu_{\mf{C}})$, $q>1$. Then, for every $1\leq r<q$, $fe^{-V_{n}}\in L^{r}(\mathscr{S}'(\R^{2}),\mu_{\mf{C}})$ and 
	\begin{equation}
		\label{Eq6.8}
	\mf{F}_{r,n}(f_{m} \, e^{-V_{n,m}})_{m\in \N}=f \, e^{-V_{n}},
	\end{equation}
	where $f_{m}:=[\varphi_{m,n}^{\ast}]^{-1}(\mathbb{E}_{m,n}(f))$ for each $m\in\N$. Therefore, for every $f\in L^{q}(\mu_{\mf{C}})$, $q>1$, the following limit holds:
	$$\int_{\mathscr{S}'(\R^{2})}f \, e^{-V_{n}} \ \mathrm{d}\mu_{\mf{C}}=\lim_{m\to\infty}\int_{H_{m,n}^{a}}f_{m} \ e^{-V_{n,m}} \ \mathrm{d}\mu_{m,n}.$$
\end{theorem}
\begin{proof}
	As in the proof of Theorem \ref{Th6.2}, we will omit the subscript $n$ in the notation $\varphi_{m,n}$. The first statement follows by simply applying Hölder's inequality:
	\begin{equation}
		\|f e^{-V_{n}}\|_{L^{r}(\mu_{\mf{C}})} \leq \|f\|_{L^{q}(\mu_{\mf{C}})} \, \|e^{-V_{n}}\|_{L^{p}(\mu_{\mf{C}})}, \quad \frac{1}{r}=\frac{1}{p}+\frac{1}{q},
	\end{equation}
     and noting that $e^{-V_{n}}\in L^{p}(\mathscr{S}'(\R^{2}),\mu_{\mf{C}})$ for each $p\geq 1$ by the lower bound in the proof of Theorem \ref{Th6.2}. To prove identity \eqref{Eq6.8}, we apply a similar argument:
     \begin{align*}
     	\mc{T}_{m}:= & \ \|\varphi_{m}^{\ast}(f_{m} e^{-V_{n,m}}) - f e^{-V_{n}}\|_{L^{r}(\mu_{\mf{C}})}
     	 \leq  \ \|\varphi_{m}^{\ast}(e^{-V_{n,m}})(\varphi_{m}^{\ast}(f_{m})-f)\|_{L^{r}(\mu_{\mf{C}})} \\
     	& + \|f (\varphi_{m}^{\ast}(e^{-V_{n,m}})-e^{-V_{n}})\|_{L^{r}(\mu_{\mf{C}})} \leq \|\varphi_{m}^{\ast}(e^{-V_{n,m}})\|_{L^{p}(\mu_{\mf{C}})} \, \|\varphi_{m}^{\ast}(f_{m})-f\|_{L^{q}(\mu_{\mf{C}})} \\
     	& +\|f\|_{L^{q}(\mu_{\mf{C}})} \, \|\varphi_{m}^{\ast}(e^{-V_{n,m}})-e^{-V_{n}}\|_{L^{p}(\mu_{\mf{C}})}.
     \end{align*}
 From the definition of $(f_{m})_{m\in\N}$, it follows that $\|\varphi_{m}^{\ast}(f_{m})-f\|_{L^{q}(\mu_{\mf{C}})}\to 0$ as $m\to\infty$. Analogously, from \eqref{Efe}, it follows $\|\varphi_{m}^{\ast}(e^{-V_{n,m}})-e^{-V_{n}}\|_{L^{p}(\mu_{\mf{C}})}\to 0$ as $m\to\infty$. Finally, as the sequence $(\varphi_{m}^{\ast}(e^{-V_{n,m}}))_{m\in\N}$ converges, its $L^{p}$-norm is uniformly bounded. Therefore, we obtain that $\mc{T}_{m}\to 0$ as $m\to\infty$. This concludes the proof.
\end{proof}
\noindent Consequently, applying Theorem \ref{th5.3}, given $h_{1},\cdots,h_{k}\in\mathscr{S}(\R^{2})$, and writing $F(\phi)=\phi(h_{1})\cdots\phi(h_{k})$ and $F_{n,m}=[\varphi^{\ast}_{m,n}]^{-1}(\mathbb{E}_{m,n}(F))$,
the corresponding smeared correlation functions can be computed via a finite-dimensional approximations as
\begin{equation*}
\mf{S}_{k}(h_{1},\cdots,h_{k}):=\frac{1}{Z}\lim_{n\to\infty}\int_{\mathscr{S}'(\R^{2})}\phi(h_{1})\cdots\phi(h_{k}) \ e^{-V_{n}} \, \mathrm{d}\mu_{\mf{C}}(\phi)
\end{equation*}
\begin{equation*}
=\frac{1}{Z}\lim_{n\to\infty}\lim_{m\to\infty}\int_{H^{a}_{m,n}} F_{n,m}(\mathbf{z}) \, \exp\left(-\frac{1}{2^{2m}}\sum_{\ell\in \mf{N}_{m}}\normord{ z^{k}_{\ell}} g(\ell)\right) \, \mathrm{d}\mu_{m,n}(\mathbf{z}).
\end{equation*}The outer limit follows from \eqref{LimE} and H\"older's inequality, since $F$ has all finite Gaussian moments.

This approach is different to the usual formal perturbation theory \cite[\S 8.4]{GJ} where one expands $e^{-V}$ in power series and computes the (interaction-free) integrals in each term by using Feymnan diagrams. This approach is expected to be useful for the computation of explicit correlations and will be the subject of future work.

\section{Application II: Unitary representations}\label{Se6}

In this section we use the abstract setting developed along Section \ref{Se3} to obtain certain unitary representations. More precisely, we set out to answer the following question. Given a family of unitary representations on $L^{2}(X_{i},\mu_{i})$, can we construct an unitary representation on $L^{2}(X_{\infty},\mu_{\infty})$ by using the identification of Theorem \ref{T3.2}? Before presenting the main results, it is useful to introduce some technical aspects. For the sake of simplicity, we will only work with the directed set $\N$ instead of a general directed set $I$. Along this section, we consider complex $L^{2}$-spaces and endow the unitary groups with the strong operator topology. Given a Hilbert space $H$, we denote by $\mc{U}(H)$ the space of unitary operators $T: H\to H$. It forms a topological group under the operator composition $\circ$. The category of topological groups will be subsequently denoted by $\mathfrak{GTop}$. Let $\mf{F}:\N\to \mf{M}$, $\mf{F}\equiv (X_{n},\mu_{n},\pi_{nm})$, be a convergent diagram in $\mathfrak{M}$ with limit $(X_{\infty},\mu_{\infty},\pi_{n})$. The isometric isomorphism $\mf{I}_{2}: \mathscr{L}^{2}(\mf{F})\to L^{2}(\mu_{\infty})$ established in Theorem \ref{T3.2} induces the following isomorphism on $\mathfrak{GTop}$,
\begin{equation*}
	\Psi: \mc{U}(\mathscr{L}^{2}(\mf{F}))\longrightarrow \mc{U}(L^{2}(\mu_{\infty})), \quad T\mapsto \mf{I}_{2}\circ T \circ \mf{I}_{2}^{-1}.
\end{equation*}
Consider the following subgroup of the direct product,
\begin{equation*}
	\mc{U}(L^{2}\mf{F}):=\Big\{(T_{n})_{n}\in \prod_{n\in \N}\mc{U}(L^{2}(\mu_{n})): \exists \ell\in \N, \, \pi_{nm}^{\ast}\circ T_{n} = T_{m} \circ \pi^{\ast}_{nm}, \, \ell \leq n\leq m\Big\},
\end{equation*}
endowed with the product topology. Naturally, the condition in the definition of $\mc{U}(L^{2}\mf{F})$ is equivalent to the commutativity in the large of the diagram of Figure \ref{F10.0}. 
\begin{figure}[h!]
	\begin{tikzcd}[row sep=large, column sep=large]
		L^{2}(\mu_{n}) \arrow{r}{T_{n}} \arrow{d}{\pi^{\ast}_{nm}} & L^{2}(\mu_{n}) \arrow{d}{\pi^{\ast}_{nm}} \\
		L^{2}(\mu_{m}) \arrow{r}{T_{m}} &  L^{2}(\mu_{m})
	\end{tikzcd}
	\caption{Diagram VI}
	\label{F10.0}
\end{figure}

\noindent The group $\mc{U}(L^{2}\mf{F})$ can be seen as the set of families whose restrictions to a tail of $\N$ are \textit{unitary} natural transformations of the restricted functor $L^{2}\mf{F}: \N \to \mf{Ban}$. The definition of $\mc{U}(L^{2}\mf{F})$ enables to define the following group homomorphism,
\begin{equation}
	\Phi: \mc{U}(L^{2}\mf{F})\longrightarrow \mc{U}(\mathscr{L}^{2}(\mf{F})), \quad \Phi(T_{n})_{n}[(f_{n})_{n}]:=(T_{n}(f_{n}))_{n}.
\end{equation}
By applying the isomorphism $\Psi$, we obtain the group homomorphism,
\begin{equation}
	\hat\Phi: \mc{U}(L^{2}\mf{F})\longrightarrow \mc{U}(L^{2}(\mu_{\infty})), \quad \hat{\Phi}:=\Psi\circ \Phi.
\end{equation}
Let $G$ be a topological group and $H$ a Hilbert space. By a (strongly continuous) unitary representation, we mean a group homomorphism $\rho: G\to \mc{U}(H)$ such that $g\mapsto \rho(g)[u]$ is continuous for every $u\in H$.
The first result of this section reads as follows.

\begin{theorem}
	\label{ThP}
	Let $(X_{n},\mu_{n},\pi_{nm})$ be a convergent diagram in $\mathfrak{M}$ with limit $(X_{\infty},\mu_{\infty},\pi_{n})$ and let $(G_{n})_{n\in\N}$ be a family of topological groups. Let $G$ be a subgroup of the product $\prod_{n\in \N}G_{n}$ endowed with the product topology and $(\rho_{n})_{n\in \N}$ be a family of strongly continuous unitary representations,
	$$\rho_{n}: G_{n} \longrightarrow \mc{U}(L^{2}(\mu_{n})),$$
	such that there exists $\ell\in \N$ for which
	\begin{equation}
		\label{CNA}
		\pi_{nm}^{\ast}\circ \rho_{n}(g_{n})=\rho_{m}(g_{m})\circ \pi_{nm}^{\ast}, \quad \text{ for each } \ell \leq n\leq m, \text{ and each } (g_{n})_{n\in \N}\in G.
	\end{equation}
    In other words, the following diagram commutes in the large:
    \begin{equation*}
    	\begin{tikzcd}[row sep=large, column sep=large]
    		L^{2}(\mu_{n}) \arrow{r}{\rho_{n}(g_{n})} \arrow{d}{\pi^{\ast}_{nm}} & L^{2}(\mu_{n}) \arrow{d}{\pi^{\ast}_{nm}} \\
    		L^{2}(\mu_{m}) \arrow{r}{\rho_{m}(g_{m})} &  L^{2}(\mu_{m})
    	\end{tikzcd}
    \end{equation*}
	Then the map
	\begin{equation*}
		\rho: G\longrightarrow \mc{U}(L^{2}(\mu_{\infty})), \quad (g_{n})_{n\in \N} \mapsto \hat{\Phi}(\rho_{n}(g_{n}))_{n\in \N},
	\end{equation*}
	is an unitary representation.
\end{theorem}

\begin{proof}
	Hypothesis \eqref{CNA} implies that $(\rho_{n}(g_{n}))_{n\in\N}\in \mc{U}(L^{2}\mf{F})$ for each $(g_{n})_{n\in \N}\in G$. Hence $\rho$ is well defined. Moreover it is a group homomorphism since
	\begin{align*}
		\rho((g_{n}\cdot h_{n})_{n\in \N}) & =\hat{\Phi}(\rho_{n}(g_{n}\cdot h_{n}))_{n\in \N}=\hat{\Phi}(\rho_{n}(g_{n})\circ \rho_{n}(h_{n}))_{n\in \N} \\
		& = \hat{\Phi}[(\rho_{n}(g_{n}))_{n\in \N}\circ (\rho_{n}(h_{n}))_{n\in \N}]\\
		& =\hat{\Phi}(\rho_{n}(g_{n}))_{n\in \N}\circ \hat{\Phi}(\rho_{n}(h_{n}))_{n\in \N}=\rho(g_{n})_{n\in \N}\circ \rho(h_{n})_{n\in \N},
	\end{align*}
	where we have used the homomorphism property of $\rho_{n}$ and $\hat{\Phi}$, respectively. It remains to prove its strong continuity. Fix $h=(h_{n})_{n}\in G$, $f\in L^{2}(\mu_\infty)$ and $\varepsilon>0$. Write $f=\mf{I}_{2}(f_{n})_{n}$. By the co-Cauchy property, choose $k\geq\ell$ such that $\|\pi^{\ast}_{kn}(f_{k})-f_{n}\|_{L^{2}(\mu_{n})}<\varepsilon/8$ for all $n\geq k$.
By the strong continuity of $\rho_{k}$, there is an open neighbourhood $U$ of $h_{k}$ in $G_{k}$ such that
$$a\in U\quad\Longrightarrow\quad
\|\rho_{k}(a)[f_{k}]-\rho_{k}(h_{k})[f_{k}]\|_{L^{2}(\mu_{k})}<\varepsilon/4.$$
The set $\mc{V}=\{g=(g_{n})_{n}\in G\mid g_{k}\in U\}$ is an open neighbourhood of $h$ in the product topology. For every $g\in\mc{V}$ and $n\geq k$, it follows from \eqref{CNA} that
\begin{align*}
&\|\rho_{n}(g_{n})[\pi^{\ast}_{kn}(f_{k})]-\rho_{n}(h_{n})[\pi^{\ast}_{kn}(f_{k})]\|_{L^{2}(\mu_{n})}\\
&\quad=\|\pi^{\ast}_{kn}(\rho_{k}(g_{k})[f_{k}]-\rho_{k}(h_{k})[f_{k}])\|_{L^{2}(\mu_{n})}\\
&\quad=\|\rho_{k}(g_{k})[f_{k}]-\rho_{k}(h_{k})[f_{k}]\|_{L^{2}(\mu_{k})}
<\varepsilon/4.
\end{align*}
By unitarity, for all $n\geq k$ we infer that
\begin{align*}
\|\rho_{n}(g_{n})[f_{n}]-\rho_{n}(h_{n})[f_{n}]\|_{L^{2}(\mu_{n})}
&\leq2\|f_{n}-\pi^{\ast}_{kn}(f_{k})\|_{L^{2}(\mu_{n})}\\
&\quad+\|\rho_{n}(g_{n})[\pi^{\ast}_{kn}(f_{k})]-\rho_{n}(h_{n})[\pi^{\ast}_{kn}(f_{k})]\|_{L^{2}(\mu_{n})}
<\varepsilon/2.
\end{align*}
Hence, for every $g\in\mc{V}$, we deduce
\begin{align*}
\|\rho(g_{n})_{n}[f]-\rho(h_{n})_{n}[f]\|_{L^{2}(\mu_\infty)}
&=\|\mf{I}_{2}((\rho_{n}(g_{n})[f_{n}]-\rho_{n}(h_{n})[f_{n}])_{n})\|_{L^{2}(\mu_\infty)}\\
&=\|(\rho_{n}(g_{n})[f_{n}]-\rho_{n}(h_{n})[f_{n}])_{n}\|_{\mathscr{L}^{2}}\\
&=\lim_{n\to\infty}\|\rho_{n}(g_{n})[f_{n}]-\rho_{n}(h_{n})[f_{n}]\|_{L^{2}(\mu_{n})}
\leq\varepsilon/2<\varepsilon.
\end{align*}
This proves the strong continuity of $\rho$ and concludes the proof.
\end{proof}

Let us give an example of the preceding construction applied to Gaussian measures. Let $E$ be an infinite dimensional Fréchet nuclear space and $\mu$ a Gaussian measure on $E'$ with covariance operator $\mf{C}$. Consider a complete $\mf{C}$-orthonormal system $(\xi_{n})_{n\in\N}\subset E$ and the corresponding convergent diagram $\mf{H}\equiv (H_{n}^{a},\mu_{n},\pi_{nm})$. The  Hilbert spaces $\mathscr{L}^{2}(\mf{H})$ and $L^{2}(E',\mu)$ are isometrically isomorphic via the operator $\mf{F}_{2}:\mathscr{L}^{2}(\mf{H})\to L^{2}(E',\mu)$ introduced in Theorem \ref{T4.5.2}. For each $n\in \N$, we denote by $\Sigma_{n}$ the symmetric group of degree $n$ endowed with the discrete topology. For each $n\in \N$, we consider the unitary representation,
\begin{equation*}
	\rho_{n}:\Sigma_{n}\longrightarrow \mc{U}(L^{2}(H^{a}_{n},\mu_{n})), \quad \rho_{n}(\sigma)[f]:=f(x_{\sigma(1)},x_{\sigma(2)},\cdots,x_{\sigma(n)}).
\end{equation*}
Indeed, the maps $\rho_{n}$ define unitary representations due to the rotationally invariant property of the Gaussian measure (see e.g. \cite[\S  B--11]{Y}). In order to apply $\hat\Phi$, we need to verify the commutativity in the large for a subgroup $G\leq\prod_{n\in\N}\Sigma_{n}$. Given $f\in L^{2}(H^{a}_{n},\mu_{n})$, $(\sigma_{n})_{n\in \N}\in \prod_{n\in \N}\Sigma_{n}$ and $n\leq m$, a simple computation gives
\begin{equation*}
	[\pi_{nm}^{\ast}\circ \rho_{n}(\sigma_{n})](f)=f(x_{\sigma_{n}(1)},x_{\sigma_{n}(2)},\cdots, x_{\sigma_{n}(n)}),
\end{equation*}
\begin{equation*}
	[\rho_{m}(\sigma_{m})\circ \pi^{\ast}_{nm}](f)=f(x_{\sigma_{m}(1)},x_{\sigma_{m}(2)},\cdots,x_{\sigma_{m}(n)}).
\end{equation*}
With these identities, a natural choice for $G$ is the subgroup
$$G:=\Big\{(\sigma_{n})_{n\in \N}\in \prod_{n\in\N}\Sigma_{n} \, : \, \exists N\in \N \text{ such that } \sigma_{m}|_{\{1,\cdots,n\}}=\sigma_{n}, \, N\leq n\leq m \Big\}.$$
Note that each $(\sigma_{n})_{n\in \N}\in G$ determines an element $\sigma_{M}\in \Sigma_{M}$ by simply choosing $M$ as the least number $N\in \N$ satisfying $\sigma_{m}|_{\{1,\cdots,n\}}=\sigma_{n}$ for each $N\leq n\leq m$. Let $\hat{\sigma}_{M}: \N\to \N$ be the infinite permutation such that 
$$\hat{\sigma}_{M}|_{\{1,\cdots,M\}}=\sigma_{M}, \quad \hat{\sigma}_{M}(m)=m,\quad \forall m > M.$$
We can establish the surjective group homomorphism
\begin{equation*}
	I: G \longrightarrow \Sigma_{\infty}, \quad (\sigma_{n})_{n\in \N} \mapsto \hat{\sigma}_{M},
\end{equation*}
where $\Sigma_{\infty}$ is the infinite symmetric group, that is, the inductive limit of the increasing sequence of the symmetric groups $\Sigma_{n}$. Elements of $\Sigma_{\infty}$ are bijections $\sigma: \N\to\N$ with finite support $\text{supp}(\sigma):=\{m\in \N : \sigma(m)\neq m\}$. The kernel of $I$ consists of eventually identity sequences, which act trivially under $\hat\Phi$. Hence $\hat\Phi$ descends to a unitary representation 
$\rho: \Sigma_{\infty} \to \mc{U}(L^{2}(E',\mu))$ acting by permutation on the corresponding variables. Note that in this example there is no difficulty with the continuity of $\rho$ as we are considering the discrete topology in $\Sigma_{\infty}$.
\par There is another complementary way  of constructing unitary representations from representations on $L^{2}(\mu_{n})$. To introduce it, let us consider the codiagram $\mf{U}: \N\to \mf{GTop}$, $\mf{U}\equiv (\, \mc{U}(L^{2}(\mu_{n})),\Pi_{nm})$, defined through the morphisms
$$\Pi_{nm}:\mc{U}(L^{2}(\mu_{n})) \longrightarrow \mc{U}(L^{2}(\mu_{m})), \quad \Pi_{nm}(T_{n}):=\hat{T}_{nm}\oplus I^{\perp}_{nm},$$
where for each $T_{n}\in \mc{U}(L^{2}(\mu_{n}))$ and each $m\geq n$, the operator $\hat{T}_{nm}$ is defined by
$$\hat{T}_{nm}: \pi_{nm}^{\ast}(L^{2}(\mu_{n})) \longrightarrow \pi_{nm}^{\ast}(L^{2}(\mu_{n})), \quad \pi^{\ast}_{nm}(f) \mapsto \pi^{\ast}_{nm}(T_{n}(f)),$$
and $I_{nm}^{\perp}:[\pi^{\ast}_{nm}(L^{2}(\mu_{n}))]^{\perp}\to[\pi^{\ast}_{nm}(L^{2}(\mu_{n}))]^{\perp}$ is the identity operator. Associated with this codiagram, we have the cocone $(\, \mc{U}(L^{2}(\mu_{\infty})), \Pi_{n})$, where the morphisms $\Pi_{n}$ are defined through
$$\Pi_{n}: \mc{U}(L^{2}(\mu_{n})) \longrightarrow \mc{U}(L^{2}(\mu_{\infty})), \quad \Pi_{n}(T_{n}):=\hat{T}_{n}\oplus I_{n}^{\perp},
$$
where for each $T_{n}\in \mc{U}(L^{2}(\mu_{n}))$ and each $m\geq n$, the operator $\hat{T}_{n}$ is defined by
$$\hat{T}_{n}:\pi_{n}^{\ast}(L^{2}(\mu_{n}))\longrightarrow \pi_{n}^{\ast}(L^{2}(\mu_{n})), \quad \pi_{n}^{\ast}(f)\mapsto \pi_{n}^{\ast}(T_{n}(f)),$$
and $I_{n}^{\perp} : [\pi_{n}^{\ast}(L^{2}(\mu_{n}))]^{\perp}\to [\pi_{n}^{\ast}(L^{2}(\mu_{n}))]^{\perp}$ is the identity operator. Under this notation, the second and last result of this section reads as follows.
\begin{theorem}
	\label{TPR}
	Let $(X_{n},\mu_{n},\pi_{nm})$ be a convergent diagram in $\mathfrak{M}$ with limit $(X_{\infty},\mu_{\infty},\pi_{n})$, $(V_{n})_{n\in \N}$ be a filtration of vector subspaces of a given Fr\'{e}chet space $(W,d)$ and $\rho_{n}: V_{n}\to \mc{U}(L^{2}(\mu_{n}))$ be a sequence of unitary representations such that the following diagram commutes for every $n,m\in \N$, $n\leq m$,
	\begin{equation}
		\label{EqFig}
		\begin{tikzcd}[row sep=large, column sep=large]
			V_{n} \arrow{r}{\rho_{n}} \arrow{d}{i_{nm}} & \mc{U}(L^{2}(\mu_{n})) \arrow{d}{\Pi_{nm}} \\
			V_{m} \arrow{r}{\rho_{m}} &  \mc{U}(L^{2}(\mu_{m}))
		\end{tikzcd}
	\end{equation}
	Endow $\bigcup_{n\in\N}V_{n}$ with the inductive-limit group topology of the subspace topologies on $V_{n}$. Then, there exists a unitary representation $\rho: \bigcup_{n\in \N} V_{n} \to \mc{U}(L^{2}(\mu_{\infty}))$ such that 
	$$\rho \, |_{V_{n}}=\Pi_{n}\circ \rho_{n}, \quad \text{ for every }  n\in\N.$$
\end{theorem}
\begin{proof}
	Let us define $\rho: \bigcup_{n\in \N} V_{n} \to \mc{U}(L^{2}(\mu_{\infty}))$ by the condition $\rho \, |_{V_{n}}=\Pi_{n}\circ \rho_{n}$ for every $n\in\N$ and let us see that this is well defined. Let $v\in V_{n} \leq V_{m}$, $n \leq m$, then it follows from the cocone property of $(\,\mc{U}(L^{2}(\mu_{\infty})),\Pi_{n})$ and the commutativity of the diagram \eqref{EqFig} that
	$$(\Pi_{n}\circ \rho_{n}) (v) = (\Pi_{m}\circ\Pi_{nm}\circ \rho_{n})(v)=(\Pi_{m}\circ \rho_{m}\circ i_{nm})(v)=(\Pi_{m}\circ \rho_{m})(v).$$
	Therefore $\rho$ is well defined. Now, we proceed to show the strong continuity of $\rho$. Fix $n\leq m$ and $w\in V_{n}$. Let $\varepsilon>0$, $f\in L^{2}(\mu_{\infty})$. Then, by the strong continuity of $\rho_{m}$ at $w$, there exists $\delta>0$ such that 
	$$v, w \in V_{n} \leq V_{m}, \, d(v,w)<\delta \quad \Longrightarrow \quad \|\rho_{m}(v)[P_{m}(f)]-\rho_{m}(w)[P_{m}(f)]\|_{L^{2}(\mu_{m})}<\varepsilon,$$
	where $P_{m}:  L^{2}(\mu_{\infty})\to \pi_{m}^{\ast}(L^{2}(\mu_{m}))\simeq L^{2}(\mu_{m})$ is the orthogonal projection. Hence, given $v,w\in V_{n}\leq V_{m}$, $n\leq m$, such that $d(v,w)<\delta$, we obtain
	\begin{align*}
		\|\rho(v)[f]-\rho(w)[f]\|_{L^{2}(\mu_{\infty})} & = \| (\Pi_{n}\circ \rho_{n}(v))[f]-(\Pi_{m}\circ\rho_{m}(w))[f]\|_{L^{2}(\mu_{\infty})} \\
		& =\|(\Pi_{m}\circ\rho_{m}(v))[f]-(\Pi_{m}\circ\rho_{m}(w))[f]\|_{L^{2}(\mu_{\infty})} \\
		& = \|\hat{\rho}_{m}(v)(P_{m}(f))-\hat{\rho}_{m}(w)(P_{m}(f))\|_{L^{2}(\mu_{\infty})} \\
		& =\|\rho_{m}(v)[P_{m}(f)]-\rho_{m}(w)[P_{m}(f)]\|_{L^{2}(\mu_{m})}<\varepsilon.
	\end{align*}
	Taking $m=n$, we obtain the strong continuity of $\rho|_{V_{n}}$ for every $n\in\N$. The universal property of the inductive-limit group topology then gives the strong continuity of $\rho$ and concludes the proof.
\end{proof}

Let us give a related construction for Gaussian measures. As before, let $E$ be an infinite dimensional Fréchet nuclear space and $\mu$ a Gaussian measure on $E'$ with covariance operator $\mf{C}$. Consider a complete $\mf{C}$-orthonormal system $(\xi_{n})_{n\in\N}\subset E$ and the corresponding convergent diagram $\mf{H}\equiv (H_{n}^{a},\mu_{n},\pi_{nm})$. For each $n\in \N$, we denote by $O(\R^{n})$ the orthogonal group in dimension $n$. Along this discussion we will identify $O(\R^{n})\simeq O(H^{a}_{n})$ for each $n\in \N$ and we will consider each $O(\R^{n})$ as a subspace of the space of bounded linear operators of $\ell^{2}$, $\mc{L}(\ell^{2})$. To visualize this embedding, we consider a complete orthogonal basis of $\ell^{2}$, $(e_{n})_{n\in\N}$, and for each $n\in \N$, the orthogonal direct sum decomposition 
$$\ell^{2}=W_{n}\oplus W_{\infty-n}, \quad W_{n}:=\spann\{e_{i} \, : \, i\in\{1,\cdots,n\}\}, \, W_{\infty-n}:=\overline{\spann\{e_{i}\, : \, i>n\}}.$$
we define the embedding as:
\begin{equation*}
	i_{n}: O(\R^{n})\hookrightarrow \mc{L}(\ell^{2}), \quad i_{n}(T):=T\oplus I_{\infty-n},
\end{equation*}
where $I_{\infty-n}:W_{\infty-n}\to W_{\infty-n}$ is the identity operator.
 Due to the rotationally invariant property of the Gaussian measure, we can consider the unitary representations,
\begin{equation*}
	\rho_{n}:O(\R^{n})\longrightarrow \mc{U}(L^{2}(H^{a}_{n},\mu_{n})), \quad \rho_{n}(T)[f]:=f\circ T^{-1}.
\end{equation*}
It is classical the fact that this representations are strongly continuous (see e.g. Folland \cite[Ch. 3]{Ff}). Using the Gaussian tensor decomposition, let us define $Q_{nm}(U)=U\otimes I$ on $L^{2}(\mu_{m})\simeq L^{2}(\mu_{n})\otimes L^{2}(\gamma_{m-n})$, where $\gamma_{m-n}$ is the standard Gaussian measure. Let us show that the diagram of Figure \ref{F10} commutes where $i_{nm}:O(\R^{n})\to O(\R^{m})$ are defined by $i_{nm}(T):=T\oplus I_{m-n}$. Here, $I_{m-n}:\R^{m-n}\to\R^{m-n}$ denotes the identity operator.
\begin{figure}[h!]
		\begin{tikzcd}[row sep=large, column sep=large]
		O(\R^{n}) \arrow{r}{\rho_{n}} \arrow{d}{i_{nm}} & \mc{U}(L^{2}(\mu_{n})) \arrow{d}{Q_{nm}} \\
		O(\R^{m}) \arrow{r}{\rho_{m}} &  \mc{U}(L^{2}(\mu_{m}))
	\end{tikzcd}
	\caption{Diagram VII}
	\label{F10}
\end{figure}

\noindent Let $f\in L^{2}(\mu_{m})$, $n\leq m$ and $P_{nm}:L^{2}(\mu_{m})\to \pi^{\ast}_{nm}(L^{2}(\mu_{n}))\simeq L^{2}(\mu_{n})$ the orthogonal projection onto $L^{2}(\mu_{n})$. Put $A=T^{-1}\oplus I_{m-n}$. On the one hand, we have
\begin{align*}
	(Q_{nm}\circ \rho_{n})(T)[f] & =(\rho_{n}(T)\otimes I) \, [P_{nm}(f)+(I_{m}-P_{nm})(f)] \\
	& =P_{nm}(f)\circ A\oplus ((I_{m}- P_{nm})(f))\circ A.
\end{align*}
On the other hand,
\begin{align*}
	(\rho_{m}\circ i_{nm})(T)[f] & =\rho_{m}(T\oplus I_{m-n})[P_{nm}(f)+(I_{m}-P_{nm})(f)] \\
	& = (P_{nm}(f)+(I_{m}-P_{nm})(f))\circ A \\
	& =P_{nm}(f)\circ A\oplus ((I_{m}- P_{nm})(f))\circ A.
\end{align*}
Consequently, the diagram of Figure \ref{F10} commutes. By Gaussian second quantization on the Hilbert completion of $(E,\mf{C})$, we obtain $\rho(T)=\Gamma[T]$ and $\rho(i_{n}(T))=\rho_{n}(T)\otimes I$. On the $r$th chaos, $\|T^{\odot r}-S^{\odot r}\|\leq r\|T-S\|$. The strong continuity follows from the density of finite sums of chaos components and the unitary property. Hence we obtain a strongly continuous unitary representation $\rho: O(\infty)\to \mc{U}(L^{2}(\mu_{\infty}))$ where 
$O(\infty):=\bigcup_{n\in\N}O(\R^{n})\subset \mc{L}(\ell^{2})$ is endowed with the subspace topology induced from $\mc{L}(\ell^{2})$.

\appendix
\section{An abstract result about colimits}\label{Se5}

In this final section we give an abstract result concerning colimits in the category $\mf{Ban}$. We will stablish an abstract framework in which the identification provided by Theorem \ref{T3.2} holds. However it must be noted that in this general setting we do not have the conditional expectation as a tool to prove the surjectivity of the operator $\mf{I}$. This problem will be overcome by a direct approximation argument.
\par Let $(X_{i},\varphi_{ij})$ be a codiagram in $\mf{Ban}$ indexed by the directed set $I$. Let us denote by $\mf{F}:I\to\mf{Ban}$ the covariant functor representing $(X_{i},\varphi_{ij})$. The following definitions are analogous to the ones introduced in Section \ref{Se3}. A sequence $(x_{i})_{i\in I}\in\prod_{i\in I}X_{i}$ is said to be co-Cauchy if for every $\varepsilon>0$, there exists $\ell\in I$ such that 
\begin{equation}
	\label{coC2}
	\|\varphi_{ij}(x_{i})-x_{j}\|_{X_{j}}<\varepsilon, \quad \forall (i,j)\in I\times I,  \quad \ell \leq i\leq j.
\end{equation}
We consider the space
\begin{equation*}
	\mathscr{L}(\mf{F}):=\Big\{(x_{i})_{i\in I}\in \prod_{i\in I}X_{i} : (x_{i})_{i\in I} \text{ is co-Cauchy}\Big\}\Big\slash \sim,
\end{equation*}
where two sequences $(x_{i})_{i\in I}, (y_{i})_{i\in I}\in \prod_{i\in I}X_{i}$ are related, $(x_{i})_{i\in I}\sim(y_{i})_{i\in I}$, if by definition
\begin{equation*}
	\lim_{i\in I} \, \|x_{i}-y_{i}\|_{X_{i}}=0,
\end{equation*}
where the limit is understood in the sense of net convergence. By the elementary properties of the limit, it becomes apparent that $\sim$ is an equivalence relation. We define a norm on $\mathscr{L}(\mf{F})$ by
\begin{equation*}
	\|(x_{i})_{i\in I}\|_{\mathscr{L}}:=\lim_{i\in I} \, \|x_{i}\|_{X_{i}}, \quad (x_{i})_{i\in I}\in \mathscr{L}(\mf{F}).
\end{equation*}
The convergence of the net $(\|x_{i}\|_{X_{i}})_{i\in I}$ follows from the co-Cauchy property by a simple computation analogous to the one given in \eqref{Eqp}.
\par Under these definitions we can state the main result of this section.

\begin{theorem}
	\label{T3.22}
	Let $(X_{i},\varphi_{ij})$ be a codiagram in $\mathfrak{Ban}$ with colimit $(X,\phi_{i})$, where $I$ is a directed set. Then, $X$ is isometrically isomorphic to $\mathscr{L}(\mf{F})$. Symbolically, 
	\begin{equation*}
		X\simeq\mathscr{L}(\mf{F}).
	\end{equation*}
\end{theorem}
\begin{proof}[Proof of Theorem \ref{T3.22}]
	Let us consider the operator
	\begin{equation*}
		\mf{I}:\mathscr{L}(\mf{F})\to X, \quad (x_{i})_{i\in I}\mapsto \lim_{i\in I} \phi_{i}(x_{i}).
	\end{equation*}
	The first part of the proof follows the same steps as the proof of Theorem \ref{T3.2}. Let us start by showing that the operator $\mf{I}$ is well defined. Let $(x_{i})_{i\in I}\in\mathscr{L}(\mf{F})$. For any given $\varepsilon>0$, taking $\ell\in I$ satisfying \eqref{coC2}, we obtain
	\begin{align*}
		\|\phi_{i}(x_{i})-\phi_{j}(x_{j})\|_{X}&=\|(\phi_{j}\circ\varphi_{ij})(x_{i})-\phi_{j}(x_{j})\|_{X}\\
		&=\|\varphi_{ij}(x_{i})-x_{j}\|_{X_{j}}<\varepsilon, \quad \forall (i,j)\in I\times I,  \ \ell \leq i\leq j.
	\end{align*}
	Hence $(\phi_{i}(x_{i}))_{i\in I}$ is a Cauchy net in $X$, from which follows its convergence. On the other hand, given two related elements $(x_{i})_{i\in I}\sim (y_{i})_{i\in I}$, we have $\mf{I}(x_{i})_{i\in I}=\mf{I}(y_{i})_{i\in I}$ since
	\begin{align*}
		\|\mf{I}(x_{i})_{i\in I}-\mf{I}(y_{i})_{i\in I}\|_{X}&=\lim_{i\in I}\|\phi_{i}(x_{i})-\phi_{i}(y_{i})\|_{X}\\
		&=\lim_{i\in I}\|x_{i}-y_{i}\|_{X_{i}}=0.
	\end{align*}
	This proves that $\mf{I}$ is well defined. By an analogous computation we obtain that $\mf{I}$ is an isometry:
	\begin{equation*}
		\|\mf{I}(x_{i})_{i\in I}\|_{X}=\lim_{i\in I}\|\phi_{i}(x_{i})\|_{X}=\lim_{i\in I}\|x_{i}\|_{X_{i}}=\|(x_{i})_{i\in I}\|_{\mathscr{L}}.
	\end{equation*}
	Finally, let us prove that $\mf{I}$ is onto. Let $x\in X$ and put
\begin{equation*}
 d_{i}=\operatorname{dist}(x,\phi_{i}(X_{i})),\qquad i\in I.
\end{equation*}
Since $\bigcup_{i\in I}\phi_{i}(X_{i})$ is dense in $X$ and $\phi_{i}(X_{i})\subseteq\phi_{j}(X_{j})$ for $i\leq j$, we have $d_{i}\to0$.
For every $i\in I$, choose $x_{i}\in X_{i}$ such that
\begin{equation*}
 \|x-\phi_{i}(x_{i})\|_{X}\leq2d_{i}.
\end{equation*}
If $d_{i}>0$, this follows from the definition of the infimum. If $d_{i}=0$, it follows from the fact that $\phi_{i}(X_{i})$ is closed, since $X_{i}$ is Banach and $\phi_{i}$ is an isometry. Consequently $\phi_{i}(x_{i})\to x$ in $X$. Moreover, for $i\leq j$,
\begin{align*}
 \|\varphi_{ij}(x_{i})-x_{j}\|_{X_{j}}
 &=\|\phi_{i}(x_{i})-\phi_{j}(x_{j})\|_{X}\\
 &\leq\|\phi_{i}(x_{i})-x\|_{X}+\|x-\phi_{j}(x_{j})\|_{X}
 \leq2d_{i}+2d_{j}.
\end{align*}
Therefore $(x_{i})_{i\in I}$ is co-Cauchy and $\mf{I}(x_{i})_{i\in I}=x$. This shows that $\mf{I}$ is onto and concludes the proof.
\end{proof}
To finalize this section we prove an analogue of Theorem \ref{T3.3} in this general setting. Let us define the isometries $(\psi_{i})_{i\in I}$, $\psi_{i}:X_{i}\to \mathscr{L}(\mf{F})$, by
\begin{equation*}
	\psi_{i}(x_{i})=(y_{j})_{j\in I}, \quad y_{j}:=\left\{\begin{array}{ll}
		\varphi_{ij}(x_{i}) & \text{ if } i\leq j \\
		0 & \text{ else }
	\end{array}\right.
\end{equation*}
An easy computation shows that the pair $(\mathscr{L}(\mf{F}),\psi_{i})$ defines a cocone in $\mathfrak{Ban}$.
\begin{theorem}
	\label{T3.32}
	Let $\mf{F}\equiv(X_{i},\varphi_{ij})$ be a codiagram in $\mathfrak{Ban}$ with colimit $(X,\phi_{i})$, where $I$ is a directed set. Then the cocone $(\mathscr{L}(\mf{F}),\psi_{i})$ defines a realization of the colimit of $(X_{i},\varphi_{ij})$. 
\end{theorem}
\begin{proof}
	Since the colimit is unique up to a unique isomorphism on the category of cocones, it is enough to prove that the cocone $(\mathscr{L}(\mf{F}),\psi_{i})$ is isomorphic to $(X,\phi_{i})$ in the category of cocones $\mf{K}^{\ast}(\mf{F})$. By the proof of Theorem \ref{T3.22}, the map 
	\begin{equation*}
		\mathfrak{I}:\mathscr{L}(\mf{F})\longrightarrow X, \quad (x_{i})_{i\in I}\mapsto \lim_{i\in I}\phi_{i}(x_{i}),
	\end{equation*}
	defines an isometric isomorphism. To prove that it defines also an isomorphism of cocones, we need to verify the commutativity of the diagram of Figure \ref{F92} for each $i,j\in I$.
	
	\begin{figure}[h!]
		\[
		\xymatrix@C+1em@R+1em{ 
			X_{i} \ar^{\varphi_{ij}}[rr] \ar^{\psi_{i}}[dr] \ar@/_1em/_{\phi_{i}}[ddr] & & X_{j} \ar_{\psi_{j}}[dl] \ar@/^1em/^{\phi_{j}}[ddl] \\
			&\mathscr{L}(\mf{F}) \ar^{\mathfrak{I}}[d] & \\
			&  X&
		}
		\]
		\caption{Diagram VIII}
		\label{F92}
	\end{figure}
	
	\noindent Take $x_{i}\in X_{i}$, then by a simple computation we obtain
	\begin{equation*}
		(\mathfrak{I}\circ \psi_{i})(x_{i})=\lim_{r}h_{r}, \quad h_{r}:=\left\{
		\begin{array}{ll}
			\phi_{i}(x_{i}) & \text{ if } i\leq r \\
			0 & \text{ else, } 
		\end{array}
		\right.
	\end{equation*}
	and thus $(\mathfrak{I}\circ \psi_{i})(x_{i})=\phi_{i}(x_{i})$. Hence the diagram of Figure \ref{F92} commutes and the proof is concluded.
\end{proof}

\end{document}